\documentclass{jie-public}

\usepackage{amssymb}
\usepackage{amsthm}

\usepackage{libertine}

\usepackage{amsmath,amsfonts}
\usepackage{hyperref}
\usepackage{mathtools}
\usepackage{multirow}
\usepackage{subfigure}
\usepackage[capitalise]{cleveref}
\crefname{figure}{figure}{figures}
\usepackage[toc,page]{appendix}
\usepackage{siunitx}
\usepackage{xcolor}
\usepackage{enumitem}

\newtheorem{thm}{Theorem}
\newtheorem{pps}[thm]{Proposition}

\newcommand{\norm}[1]{{\left|#1\right|}}
\newcommand{\normm}[1]{{\left\|#1\right\|}}

\newcommand{\kh}[1]{{\left(#1\right)}}
\newcommand{\zkh}[1]{{\left[#1\right]}}
\newcommand{\dkh}[1]{{\left\{#1\right\}}}
\newcommand{\DEF}{\coloneqq}

\newcommand{\bE}{{\mathbb{E}}}
\newcommand{\bR}{{\mathbb{R}}}
\newcommand{\bS}{{\mathbb{S}}}
\newcommand{\bM}{{\mathbb{M}}}
\newcommand{\bW}{{\mathbb{W}}}
\newcommand{\bN}{{\mathbb{N}}}

\newcommand{\cT}{{\mathcal{T}}}

\newcommand{\cS}{{\mathcal{S}}}
\newcommand{\cV}{{\mathcal{V}}}
\newcommand{\cX}{{\mathcal{X}}}
\newcommand{\cG}{{\mathcal{G}}}

\newcommand{\cK}{{\mathcal{K}}}
\newcommand{\cZ}{{\mathcal{Z}}}

\newcommand{\vc}{{\boldsymbol c}}
\newcommand{\mA}{\mathsf A}
\newcommand{\fA}{{\mathsf{A}}}
\newcommand{\fB}{{\mathsf{B}}}
\newcommand{\fG}{{\mathsf{G}}}

\newcommand{\mlc}[1]{\begin{tabular}{@{}c@{}}#1\end{tabular}}

\newcommand{\cov}{\text{Cov}}

\newcommand{\code}[1]{{\fontfamily{qcr}\selectfont #1}}

\title{Fast Multiscale Functional Estimation in Optimal EMG Placement for Robotic Prosthesis Controllers}

\author{Jin Ren$^\ast$}
\address{Department of Mathematics \& Statistics, Old Dominion University, Old Dominion University, USA}
\email{jren@odu.edu}

\author{Guohui Song$^\ast$}
\address{Department of Mathematics \& Statistics, Old Dominion University, Old Dominion University, USA}
\email{gsong@odu.edu}

\author{Lucia Tabacu$^{\ast\dagger}$}
\address{Department of Mathematics \& Statistics, Old Dominion University, Old Dominion University, USA}
\email{ltabacu@odu.edu}
\thanks{$^\dagger$ Corresponding author}

\author{Yuesheng Xu$^\ast$}
\address{Department of Mathematics \& Statistics, Old Dominion University, Old Dominion University, USA}
\email{y1xu@odu.edu}

\begin{document}
\begin{abstract}
Electrocardiogram (EMG) signals play a significant role in decoding muscle contraction information for robotic hand prosthesis controllers. Widely applied decoders require large amount of EMG signals sensors, resulting in complicated calculations and unsatisfactory predictions. By the biomechanical process of single degree-of-freedom human hand movements, only several EMG signals are essential for accurate predictions. Recently, a novel predictor of hand movements adopts a multistage Sequential, Adaptive Functional Estimation (SAFE) method based on historical Functional Linear Model (FLM) to select important EMG signals and provide precise projections. 

However, SAFE repeatedly performs matrix-vector multiplications with a dense representation matrix of the integral operator for the FLM, which is computational expansive. Noting that with a properly chosen basis, the representation of the integral operator concentrates on a few bands of the basis, the goal of this study is to develop a {\it fast} Multiscale SAFE (MSAFE) method aiming at reducing computational costs while preserving (or even improving) the accuracy of the original SAFE method. Specifically, a multiscale piecewise polynomial basis is adopted to discretize the integral operator for the FLM, resulting in an {\it approximately sparse} representation matrix, and then the matrix is truncated to a sparse one. This approach not only accelerates computations but also improves robustness against noises. When applied to real hand movement data, MSAFE saves 85\%$\sim$90\% computing time compared with SAFE, while producing better sensor selection and comparable accuracy. In a simulation study, MSAFE shows stronger stability in sensor selection and prediction accuracy against correlated noise than SAFE.
\end{abstract}

{}

\maketitle
\section{Introduction}

This paper aims at developing a fast computing algorithm for the adaptive functional estimation method for robotic hand prosthesis controllers.
Robotic hand prostheses equipped with a prosthesis controller (PC), such as DEKA arm system \cite{resnik2011development,resnik2011using}, could emulate the functionality of an intact hand and assist transradial amputees (TRA) in their daily life activities. Electrodes are placed on multiple muscles of the residual limb to collect electromyogram (EMG) signals, which contain the information of muscle contraction magnitude and duration. For an able-bodied person the muscle contractions activate the tendons and the bones to produce the hand movement. For a transradial person the residual muscles can still contract and the prosthesis controller decodes the EMG signals to produce the hand prosthesis movements. However, for a transradial person many muscles are not accessible for external EMG sensors placement. The key challenges are deciding on the number/places for the EMG sensors and modeling the decoding of the EMG signals to movement.

A popular approach for prosthesis control is EMG pattern recognition, which usually uses redundant \cite{Huang2010} or high-density electrodes \cite{resnik2018pattern} to capture sufficient neural information. Classification techniques are then employed to identify the patterns of hand/wrist motions from such abundant data \cite{scheme2011electromyogram}. It has produced promising results and also brought many challenges, including the real-time processing cost of large amount of data and the extra noises introduced by additional sensors. On the other hand, a low-dimentional PC decoder with only 4 EMG signals has been shown to accurately predict wrist/hand movement in \cite{Crouch2016}. The selection of these EMG signals is based on their prior knowledge of the important muscles in the musculoskeletal structure of able-bodied (AB) subjects. However, it would be much more challenging to select relevant EMG sensors for TRA subjects since the musculoskeletal structure might be changed due to the loss of many muscles. 

A functional estimation procedure called Sequential Adaptive Functional Estimation (SAFE) has been proposed recently in \cite{Stallrich2020} to select the EMG sensors and decode the EMG signals into wrist/hand movement for TRA subjects. The statistical model proposed in SAFE uses multiple functional covariates representing recent past behavior EMG signals, whose effects can vary with the recent position (flexion or extension) of fingers and wrist, to predict the velocity or acceleration of a given movement. An adaptive group Least Absolute Shrinkage and Selection Operator (group LASSO) penalty \cite{tibshirani1996regression,yuan2006model} is employed to select the EMG sensors, then a smooth ridge regularization consisting of only the most relevant EMG sensors is used to replace the group LASSO penalty in the decoding process. It could reduce the estimation bias caused by the group LASSO penalty \cite{Stallrich2020}. It selects very few relevant EMG sensors and uses them to decode the finger/wrist movement information without sacrificing prediction accuracy. The adaptive procedure promotes the accuracy performance in sensor selection and movement estimation. However, such adaptive approach involves heavy cross validations. SAFE method uses the single-scale spline basis to estimate the predictors, which generates dense coefficient matrices. Dense matrices could dramatically drag down the speed of the computation. As experiments in \cite{Stallrich2020} illustrate, SAFE costs tens of hours in a personal computer to get final results on every single data set. This is a computational bottleneck for practical applications of the SAFE method.

To overcome this computational burden of the single-scale spline basis SAFE method, we propose in this paper a fast \textit{multiscale} numerical scheme to solve the functional linear model. It has been understood \cite{beylkin1991fast} that representation of integral operators could be numerically sparse under the so-called \textit{multiscale method}. Different from single-scale basis, multiscale basis extracts information of the integral operator from different scale, which naturally leads us to coefficient matrices concentrating at $0$. After a proper truncation, the resulting model remains precise while having sparse representation. This sparsity could help accelerating the computational process. Also by the truncated multiscale representation, we have noise partly filtered out from input data, which also boosters the multiscale method against noise.

Specially, following the idea of multiscale methods for integral equations \cite{micchelli1994using,micchelli1997reconstruction,chen2002fast}, here we apply the multiscale piecewise polynomial basis in discretization of the integral operators in the FLM model, to obtain the proposed Multiscale SAFE (MSAFE) method. Such a multiscale basis has vanishing moments and shrinking supports, which results in a coefficient matrix with decaying entries. This property enables us to approximate the coefficient matrices by a sparse one and therefore accelerates the calculation. After a proper truncation of the multiscale coefficient matrices, the computational costs can be reduced significantly comparing to the single-scale spline basis SAFE method. Also, the multiscale method is more robust against noise due to the truncation. As experiments on real data reveal, MSAFE saves 85\%$\sim$90\% of computational time as SAFE method, while providing even better sensor selection and prediction errors. In simulation study, we test SAFE and MSAFE methods with correlated data, where MSAFE outperforms single-scale method. Such idea could be extensively applied in other integral models with ease.

For a fair comparison, in this paper we make several compromises on the proposed multiscale method. On one hand, to solve the group LASSO model, SAFE applies the popular method introduced in \cite{yang2015fast} by solving a smoothing model instead. Such method gives fast solution estimations of group LASSO models with set of parameters, but could have significant deviation from the real global solutions. Actually the group LASSO model is a special case of non-smooth convex optimization problems, which were extensively studied in past decades, and many algorithms were developed to solve such kind of problems with solid convergence analysis. For this type of non-smooth convex optimization problems, one may refer to fixed-point proximity algorithms \cite{Micchelli2013proximity,li2015multi}, primal-dual algorithms \cite{esser2010general,Chambolle2011first} and alternating direction methods of multipliers \cite{esser2009applications,boyd2011distributed}. On the other hand, the fast comprehensive method to solve integral equations is considered to be the multiscale collocation method \cite{chen1999construction,chen2015multiscale}, which, in addition to multiscale basis, utilizes the multiscale collocation functionals in model discretization. Both multiscale basis and multiscale collocation functionals contribute together to an even more sparse coefficient matrix. These new concepts however will cause considerable changes to the SAFE method, and make it difficult to distinguish contributions to the final improvement among all modifications. To demonstrate the advantage of the use of the multiscale basis alone, in this paper we keep most of the original SAFE method intact, and only apply the multiscale basis to the corresponding part of the SAFE method.


To summarize, this paper contributes to SAFE method in the following aspects. We substitute the single-scale basis of the SAFE method with multiscale piecewise polynomial basis in the FLM model, which systematically generates sparse coefficient matrices after proper truncation. Also, since the widely-used \code{R} \cite{R} package \code{gglasso} \cite{gglasso} for the group LASSO model is not accepting sparse matrices, we modify corresponding functions in the package for maximum acceleration and fair comparison. These together are combined to be the proposed Multiscale SAFE (MSAFE) method and lead to the final improvement to the SAFE method.

The rest of this paper is organized as follows. We describe in Section 2 the functional linear model and the original SAFE method for EMG-based hand-movement predictors. In Section 3, we introduce the fast multiscale method, MSAFE, to solve integral models for sensor selection and movement estimation in SAFE method. Section 4 contains the application of the proposed method on the real data sets studied in \cite{Stallrich2020}. To address the robustness of the proposed method against correlation, Section 5 shows the results on simulated data with correlated noises. Conclusions of this study are drawn in Section 6. For simplicity of the presentation, we provide the mathematical derivations and technical proofs of the multiscale analysis in Appendix.

\section{EMG-Based Predictor for Hand Movements}\label{sec:problem_description}

We briefly describe in this section the operating principles of the prosthesis controller with EMG signals, and review the cutting-edge SAFE method which models the wrist/finger movement with historical FLM involving corresponding EMG signals. 


For an AB subject, the intended hand movement originates from active potentials in motor neurons in cerebral cortex. Those neural signals conduct along motor neural pathway and infuse into corresponding muscle cells, which then cause muscle contractions to accomplish the intended movement. The terminal action potential measured from the muscle fibers is defined as motor unit action potential (MUAP), which is positively correlated to the magnitude and duration of muscle contractions. The EMG sensors placed on subject's forearm could measure the sum of MUAPs across muscles, and therefore serve as effective indicators for predicting hand movements, for both AB subjects and TRAs. 
The predicted hand movements are fingers/wrist flexions and extensions in different arm postures. Due to passive forces  triggered by muscle relaxation, \cite{Stallrich2020} showed graphically  that finger movement can happen when no active EMG signal is recorded. \cite{Stallrich2020} also noticed that there are significant correlations among all EMG signals across the 30-seconds time window the data was collected. It is possible to have multiple active EMG signals when performing one instance of finger flexion and extension. Based on these findings, SAFE predicts finger/wrist movement based on the recent past EMG signals and current finger/wrist position. 


We point out that there are two important issues in designing the prosthesis controller with EMG signals: the selection of most relevant EMG sensors and the decoding of EMG signals to wrist/finger movement. There are 20 muscles of the forearm controlling various movements of wrist and fingers of the hand \cite{mitchell2019anatomy}. For a more accurate and interpretable prosthesis controller with a rapid real-time response, we have to select the most important EMG signals and have an accurate and efficient algorithm to decode them into the wrist/finger movement. Both problems would rely on an accurate quantitative model of the relation between EMG signals and wrist/finger movement. 

We will employ the flexible statistical model introduced in SAFE. In particular, the functional linear model (FLM) \cite{cardot1999functional,Cardot2003,Ferraty2012,Goldsmith2011,Mclean2014,Wu2010}, specially historical FLM \cite{malfait2003historical} would be used to describe the velocity/acceleration of wrist/finger based on recent past EMG signals. We remark that historical FLM has received many successful applications in functional regression problems. 

In \cite{malfait2003historical} the authors use the finite element method to estimate the historical functional model. This model considers a sample of curves $y_i(t)$ that can be predicted by covariate curves $x_i(s)$ with $s\in[t-\delta, t]$ and $\delta>0$ was estimated from the data. A speech production experiment is used to show the performance of the historical functional model. The data on different groups of muscles involved in the anatomy and physiology of speech was collected by EMG sensors. The curves $y_i(t)$ represent the accelerations of the center of the lower lip and the covariates $x_i(t)$ represent the EMG signal associated with the depressor labii inferior muscle. In \cite{Harezlak2007} the authors proposed a new method of estimating the historical functional model of \cite{malfait2003historical}. Their procedure combines regularization with $L^1$- and $L^2$-norm penalties of the coefficients of the neighboring basis functions. The model is then fit to a data set collected on a sample of boilermaker workers  that studies the relationship between occupational particulate matter and heart rate variability. 

In \cite{Kim2011} the recent functional linear model for sparse longitudinal data is studied. The longitudinal predictor defined only in a sliding window into the recent past $[t-\delta_1, t-\delta_2]$ for $0< \delta_2 < \delta_1 < T$ has an effect on the longitudinal response. This model is then applied to a primary biliary liver cirrhosis longitudinal data where the relationship between serum albumin concentration and prothrombin time is investigated. Historical functional models with a large number of functional and scalar covariates as in \cite{Brockhaus2017} and models with factor-specific random historical effects as in \cite{Rugamer2018} are estimated by a component-wise gradient boosting algorithm which is suitable for complex models. A fully Bayesian estimation approach based on the discrete wavelet-packet transformation was employed in \cite{Meyer2021}. 

We next briefly introduce the statistical model applied in SAFE method. Suppose we have $K$ measured and processed EMG signals and $N$ instances of measurement at different time $\{t_i\}_{i=1}^N$. For $1\le i\le N$ and $1\le k\le K$, we use $X_{ik}$ to denote the $k$-th historical EMG signal at $i$-th instance. We would like to use them to predict $y_i\DEF y(t_i)$, the response (velocity or acceleration) of the movement at time $t_i$ along with position $z_i$. The historical FLM employed in SAFE is
\begin{equation}\label{mdl:FLM_ORI}
    \bE\zkh{y_i\middle| X_{i1},X_{i2},\ldots,X_{iK},z_i}=\sum_{k=1}^K\int_{\mathcal{T}} X_{ik}\kh{\tau}\gamma_k\kh{\tau,z_i}d\tau,\qquad 1\leq i\leq N,
\end{equation}
where $\mathcal{T}\DEF[-\delta,0]$ with $\delta>0$ defines the length of historical time window, $X_{ik}(\tau)\DEF X_k(t_i+\tau)$ is the historical EMG signal at time $t_i$ and $\gamma_k$ are the unknown bivariate kernels defined on $\mathcal{T}\times\mathcal{Z}$ with $\gamma_k(\cdot,z)\in L^2(\mathcal{T})$ for any $z\in\mathcal{Z}$ and $1\leq k\leq K$. Here $\mathcal{Z}\subset\bR$ is the range of position.

We will rely on the above model to select the most important EMG sensors and predict the velocity/acceleration of movement. To select EMG sensors, a group LASSO regularization model \cite{yuan2006model} could be applied in \cref{mdl:FLM_ORI}, resulting in the following model
\begin{equation}\label{mdl:FLM_GGL}
    \begin{aligned}
        \min_{\gamma_k\in H^2}&\left\{\sum_{i=1}^N\normm{y_i-\sum_{k=1}^K \int_{\mathcal{T}} X_{ik}\kh{\tau}\gamma_k\kh{\tau,z_i}d\tau}^2\right.\\
        &\qquad\qquad\qquad\left.+\lambda\sum_{k=1}^K\sqrt{f_k\normm{\gamma_k}^2+\phi_t g_k\normm{\gamma''_{k,t}}^2+\phi_z h_k\normm{\gamma''_{k,z}}^2}\right\},
    \end{aligned}
\end{equation}
where $H^2\coloneqq W^{2,2}(\cT\times\cZ)$ is the Sobolev space of all functions possessing $L^2$ derivatives at least order 2 on $\cT\times\cZ$, $\|f\|^2\coloneqq\iint_{\cT\times\cZ} f^2(t,z)dtdz$,  $f''_t\coloneqq \partial^2f/\partial t^2$, and $f''_z\coloneqq \partial^2f/\partial z^2$ for any $f\in H^2$. The non-negative constants $f_k,g_k,h_k$ for $1\leq k\leq K$ control the penalty weights, and non-negative constants $\phi_t,\phi_z\geq0$ serve as global controllers of penalty weights on the norms of derivatives $\{\|\gamma''_{k,t}\|^2\}_{k=1}^K$ and $\{\|\gamma''_{k,z}\|^2\}_{k=1}^K$, respectively. We remark that the magnitude of estimated $\gamma_k$'s in \cref{mdl:FLM_GGL} measures the importance of the corresponding EMG signal. If the kernel $\gamma_k$ is estimated to be 0, then the corresponding $k$-th EMG signal will be considered insignificant to the movement of interest.

This idea is applied to select the most important EMG sensors. In particular, we will follow the multistage procedure proposed in SAFE method. Let $\cK^0\DEF\{1,2,\ldots, K\}$ and we first solve \cref{mdl:FLM_GGL} with $f_k=g_k=h_k=1$ for $1\le k\le K$ to get estimators $\{\hat{\gamma}^1_k\}_{k\in\cK^0}$. We then define the active variable set $\cK^1\DEF\{k\in \cK^0: \hat{\gamma}^1_k \neq 0\}$ and update the weights
\begin{align}\label{mdl:FLM_GGL_fgh}
    f_k^1 =\normm{\hat{\gamma}^1_k}^{-1}, \quad g_k^1=\normm{\kh{\hat{\gamma}^1_k}^{\prime\prime}_{t}}^{-1}, \quad h_k^1=\normm{\kh{\hat{\gamma}^1_k}^{\prime\prime}_{z}}^{-1}, \qquad \text{for all }k\in \cK^1.
\end{align}
Then we find the new estimators $\{\hat{\gamma}^2_k\}_{k\in\cK^1}$ by solving \cref{mdl:FLM_GGL}  with the active variable set $\cK^1$ and the updated weights:
\begin{align}\label{mdl:FLM_GGL_2}
      \begin{aligned}
        \min_{\gamma_k\in H^2,k\in\cK^1}&\left\{\sum_{i=1}^N\normm{y_i-\sum_{k\in \cK^1} \int_{\mathcal{T}} X_{ik}\kh{\tau}\gamma_k\kh{\tau,z_i}d\tau}^2\right.\\
        &\qquad\qquad\qquad\left.+\lambda\sum_{k\in \cK^1}\sqrt{f_k^1\normm{\gamma_k}^2+\phi_t g_k^1\normm{\gamma''_{k,t}}^2+\phi_z h_k^1\normm{\gamma''_{k,z}}^2}\right\}.
    \end{aligned}  
\end{align}
The new active variable set could be defined in a similar way: $\cK^2\DEF\{k\in \cK^1: \hat{\gamma}^2_k \neq 0\}$. Suppose $R$ stages are repeated till certain stop criteria is met, then we arrive at the selected set of EMG sensors $\cK^R$. After the most relevant sensors $\cK^R$ are determined, SAFE method suggests using the smooth ridge regression model to get the final estimation of the kernels. That is, we will find the estimator of $\gamma_k$ for each $k\in \cK^R$ by solving
\begin{align}\label{mdl:stage2}
      \begin{aligned}
        \min_{\gamma_k\in H^2,k\in\cK^R}&\left\{\sum_{i=1}^N\normm{y_i-\sum_{k\in \cK^R} \int_{\mathcal{T}} X_{ik}\kh{\tau}\gamma_k\kh{\tau,z_i}d\tau}^2\right.\\
        &\qquad\qquad\qquad\left.+\sum_{k\in \cK^R}\biggl(\phi\normm{\gamma_k}^2+\phi_t \normm{\gamma''_{k,t}}^2+\phi_z \normm{\gamma''_{k,z}}^2 \biggr)\right\}.
    \end{aligned}  
\end{align}
The regularization parameters $\phi,\phi_t,\phi_z\geq0$ are usually chosen by cross-validation method over certain candidacies. We point out that since the most relevant sensors have been chosen via \cref{mdl:FLM_GGL,mdl:FLM_GGL_2} in the first stage, we should not apply any sparse regularization in the final stage of functional estimation. The smooth regularization in \cref{mdl:stage2} could reduce the estimation bias caused by sparse penalty \cite{Leeb2015, Zhao2020}.

It is direct to observe that the major computation cost of the method comes from solving \cref{mdl:FLM_GGL}. We need to sequentially solve the same model with different parameters $f_k, g_k, h_k$ at each stage of the EMG signal selection. Moreover, cross validation is applied to choose the optimal regularization parameters $\lambda,\phi_t,\phi_z$ among candidates. These above mean that we need to solve \cref{mdl:FLM_GGL} with different constants repeatedly for a large number of times. It is necessary to develop a fast and efficient algorithm to numerically solve \cref{mdl:FLM_GGL}.

\section{Multiscale SAFE Method}
In this section we present the proposed fast multiscale SAFE (MSAFE) method to solve the models \cref{mdl:FLM_GGL,mdl:stage2}. Specially, we will employ the multiscale basis functions introduced in \cite{micchelli1994using,chen1999construction,chen2015multiscale}, which is widely used in developing fast algorithms for solving integral equations efficiently.

We remark that \cref{mdl:FLM_GGL,mdl:stage2} are minimization problems over infinite-dimensional spaces of functions, which requires discretization to solve them numerically. The original SAFE method uses tensor products of orthogonal cubic spline bases to represent the unknown kernels $\gamma_k$'s. Such full-supported single-scale basis suffers from high computational cost in two aspects. On one hand, such full-supported basis requires integral in the full domain in every calculation, which causes heavy computation when assembling coefficient matrices. On the other hand, each of the basis functions only extracts information from different part of the kernel function under the same scale. This results in coefficients of flatten distribution and therefore leads to dense matrices, which slow down the computation at later-stages.

To overcome disadvantages of the single-scale basis, here we employ a multiscale piecewise polynomial basis to discretize the integral operators in \cref{mdl:FLM_GGL,mdl:stage2}. As shown in the following sections, the multiscale basis would systematically generate sparse coefficient matrices, which therefore significantly improve the computational speed.

\subsection{Multiscale Piecewise Polynomial Basis}

We in this section briefly introduce the multiscale piecewise polynomial basis and its properties, then show in a simple numerical case that such basis could systematically yield sparse coefficient matrices. For simplicity of presentation, we leave the technical constructions and mathematical proofs of the multiscale piecewise polynomial basis in \ref{app}.

To overcome the disadvantages of full-supported single-scale basis, the multiscale basis improves in both the speed of generating coefficient matrices and the sparsity of the resulting matrices. Multiscale piecewise polynomial basis functions, analogous to wavelets, have vanishing moment and are divided into different levels. Such basis functions are orthogonal between different levels, and have exponentially shrinking support as level increases. The shrinking support of the basis functions accelerates calculations of high level coefficients and, together with the vanishing moment and orthogonality, enables the multiscale basis to capture information of the kernel in different levels. These together result in a sparse coefficient matrix. 

Specifically, we let $\bM_n^p$ be the space of all piecewise polynomials of degree no more than $p$ on $[0,1]$, with nodes at $\{i/2^n\}_{i=0}^{2^n}$. Such a multiscale piecewise polynomial space approaches space $H^2$ as level $n$ increases, and has the property that $\bM_n^p\subset\bM_{n+1}^p$, which could lead to a multilevel structured basis, as briefly described below. At the first level $W_0$, we choose $p+1$ basis polynomials of $\bM_0^p$. At next level $W_1$, we choose $p+2$ piecewise polynomials on $\bM_1^p$ that are perpendicular to $\bM_0^p$. For higher level $l\geq 2$, $W_l$ could be generated by recursively scaling and shifting functions of $W_{l-1}$:
\[
    W_l = \{\cT_iw:w\in W_{l-1},i\in\{0,1\}\},
\]
where
\[
    \kh{\cT_0f}\kh{x}\DEF f\kh{2x},\qquad \kh{\cT_1f}\kh{x}\DEF f\kh{2x-1}.
\]
It is direct to observe that every function in $W_l$ is perpendicular to $\bM_0^p$ and has a support on an interval with length no more than $2^{-l+1}$. 

We point out that the introduced multiscale piecewise polynomial basis could be used to design fast algorithms for solving integral equations, such as the Fredholm integral equation of the first kind
\begin{equation}\label{mdl:Fredholm}
    y\kh{t}=\int_0^1 K\kh{t,\tau}\gamma\kh{\tau}d\tau,
\end{equation}
when the unknown function $\gamma$ is represented by the basis functions in $W_l$. We provide a brief explanation below:
\begin{enumerate}[label=(\Roman*)]
    \item   The basis functions in $W_l$ have small supports and we only need to calculate integrals on small intervals rather than $[0,1]$ when creating coefficient matrices of \cref{mdl:Fredholm}.

    \item\label{itm:1}   The basis $W_l$ for $l\geq 1$ are all perpendicular to $\bM_0^p$, which means they have vanishing moment $p+1$. The magnitude of coefficient matrix entries will decay as $l$ increases. In particular, if kernel $K$ in \cref{mdl:Fredholm} is smooth enough, then entries of the corresponding coefficient matrix would have an exponential decay. Therefore it could be approximated by a sparse matrix. \Cref{pps:ele_est} in \ref{app} provides a theoretical justification of this property.
\end{enumerate}

\begin{figure}[!t]
    \centering
    \caption{Frequency histograms of coefficient matrix of integral equation \cref{mdl:Fredholm} with spline basis and multiscale basis, on FC1 data set from \cite{Stallrich2020}. Height of each bar indicates the frequency of values failing into the corresponding $x$-interval. The multiscale coefficients highly concentrate around 0, while single-scale spline coefficients do not.}\label{fig:hist_1}
    \hfill\subfigure[Frequency histogram of spline coefficients.]{
        \includegraphics[width=0.45\linewidth]{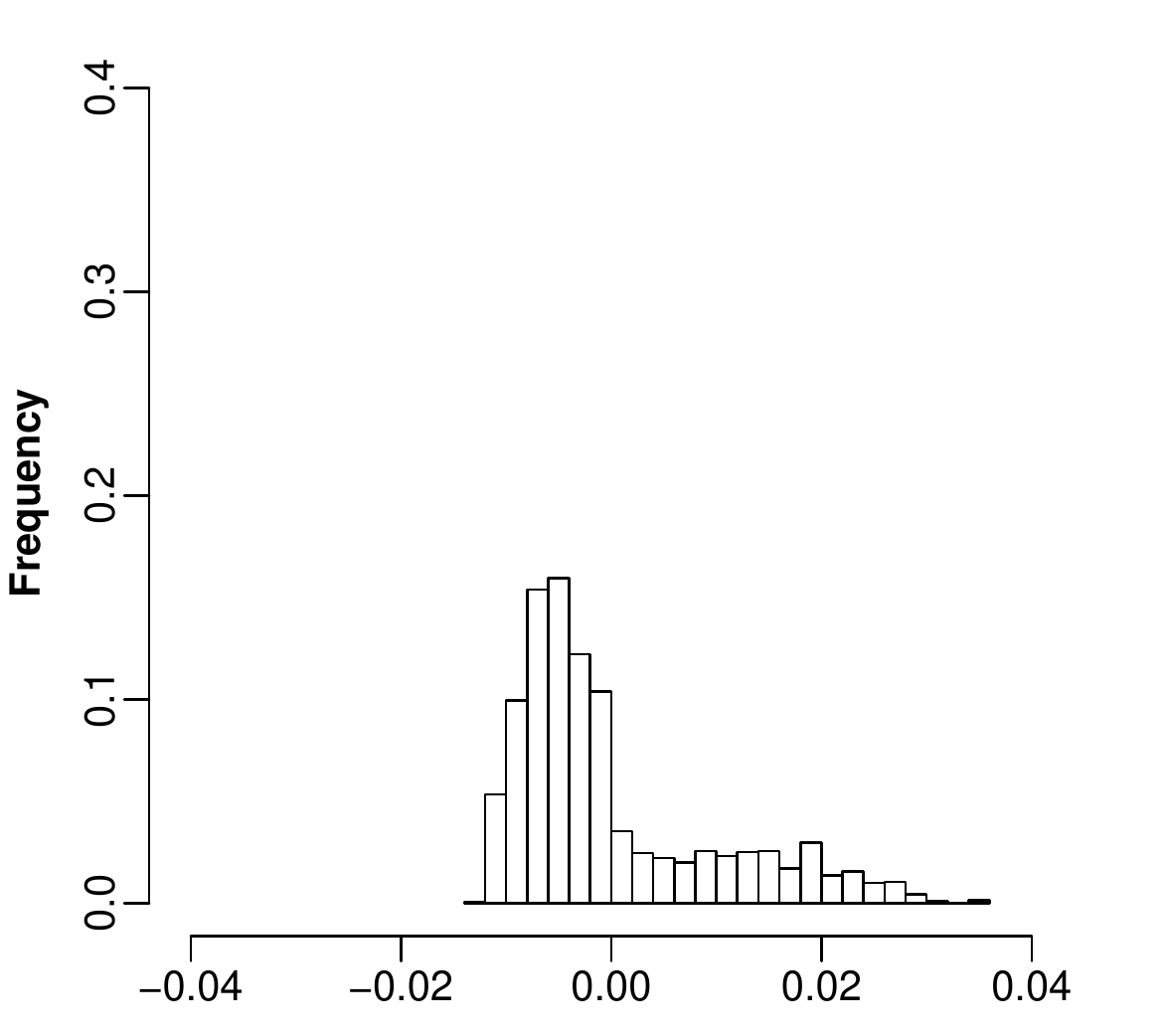}
    }\hfill
    \subfigure[Frequency histogram of multiscale coefficients.]{
        \includegraphics[width=0.45\linewidth]{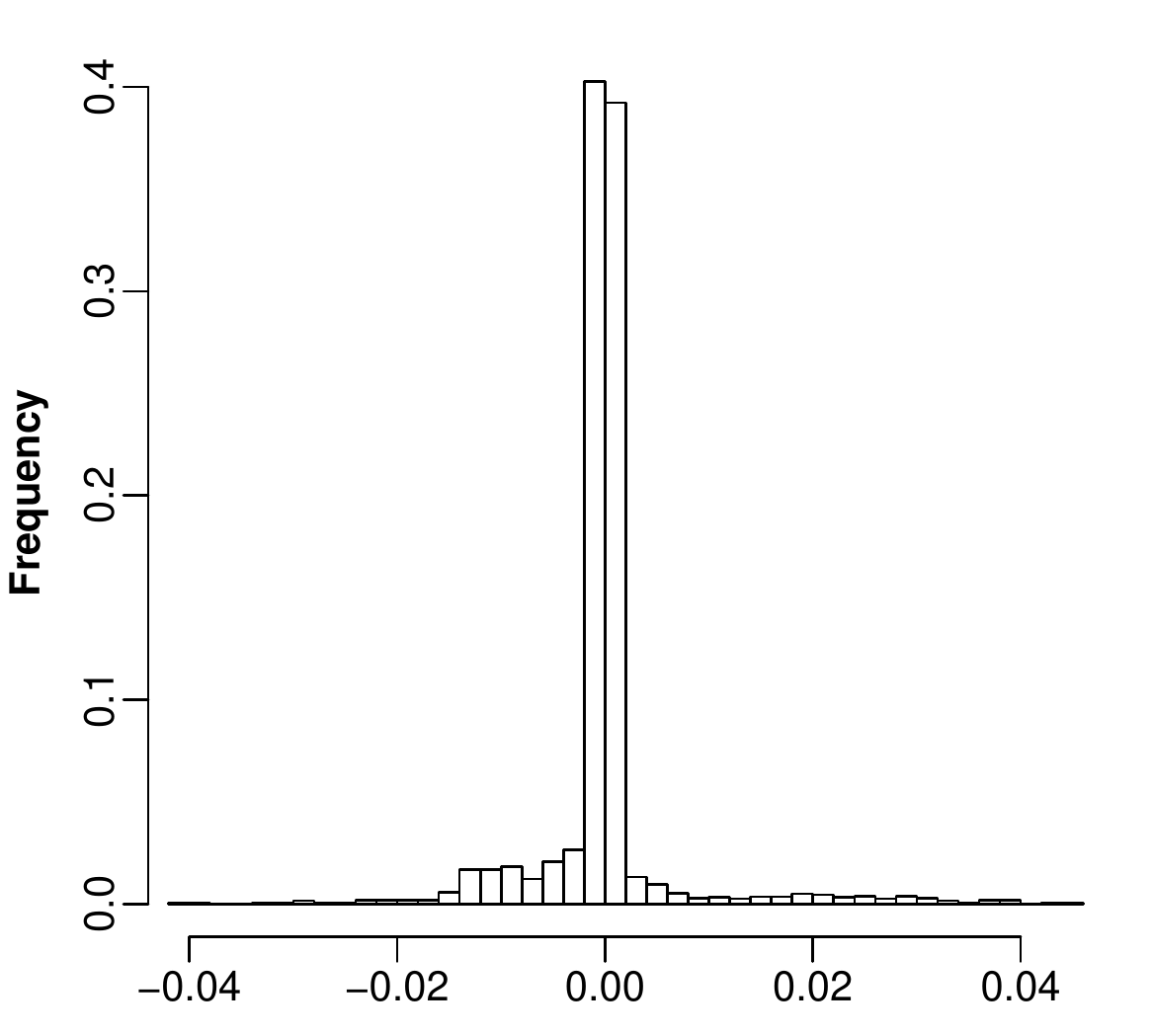}
    }\hfill
\end{figure}

Here we verify those properties with a numerical example. Suppose we would like to represent the unknown function $\gamma$ by $n$ level multiscale basis $\{w_j\}_{j=1}^{2^n(p+1)}\DEF\bigcup_{l=0}^nW_l$ and discretize \cref{mdl:Fredholm} at different sampling times $\{t_i\}_{i=1}^N$ as follows 
\begin{equation}\label{mdl:dc_multi_basis}
    y\kh{t_i}\approx \sum_{j=1}^{2^n\kh{p+1}}c_j\int_0^1 K\kh{t_i,\tau}w_j\kh{\tau}d\tau,\qquad 1\leq i\leq N.
\end{equation}
We point out that the coefficient matrix 
\begin{equation}\label{eq:A}
    A_n\DEF\left[\int_0^1 K\kh{t_i,\tau}w_j\kh{\tau}d\tau: 1\leq i\leq N, 1\leq j\leq 2^n\kh{p+1}\right]
\end{equation}
determines the computational cost of finding the coefficients $\vc\DEF(c_j)_j$. In other words, if the coefficient matrix $A_n$ is sparse, it would be much more efficient to find $\vc$. We consider $K(t,\tau)\DEF X_k(t-\delta \tau)$ with $1\leq k\leq16$ and $\{t_i\}_{i=1}^{198}$ in \cref{mdl:dc_multi_basis}, where $X_k$ for $1\leq k\leq16$ are the real EMG signals of data set FC1 from \cite{Stallrich2020}. We compare in \Cref{fig:hist_1,fig:hist_2} the sparsity of the coefficient matrix generated by the spline basis used in \cite{Stallrich2020} and the multiscale basis defined in \ref{app}. It is direct to observe from \Cref{fig:hist_1,fig:hist_2} that the coefficient matrices generated by the multiscale basis concentrates around 0, decays rapidly as the level increases and is much more sparse than those generated by the single-scale spline basis.

\begin{figure}[ht]
    \centering
    \caption{Magnitude plots of coefficient matrices \cref{mdl:Fredholm} for spline basis and multiscale basis. Columns correspond to the sampling time $\{t_i\}_{i=1}^{198}$, while rows mean different basis functions. Notice that matrices $\{X_k\}_{k=1}^{16}$ are combined in row for each basis respectively. There is a notable pattern of repeating in matrix of single-scale spline basis, which therefore has no sparse structure. For multiscale basis, most information of the kernel is extracted by the first few levels of basis functions, leading to a sparse coefficient matrix.}\label{fig:hist_2}
    \hfill\subfigure[Coefficient matrix from the spline basis.]{
        \includegraphics[width=0.45\linewidth]{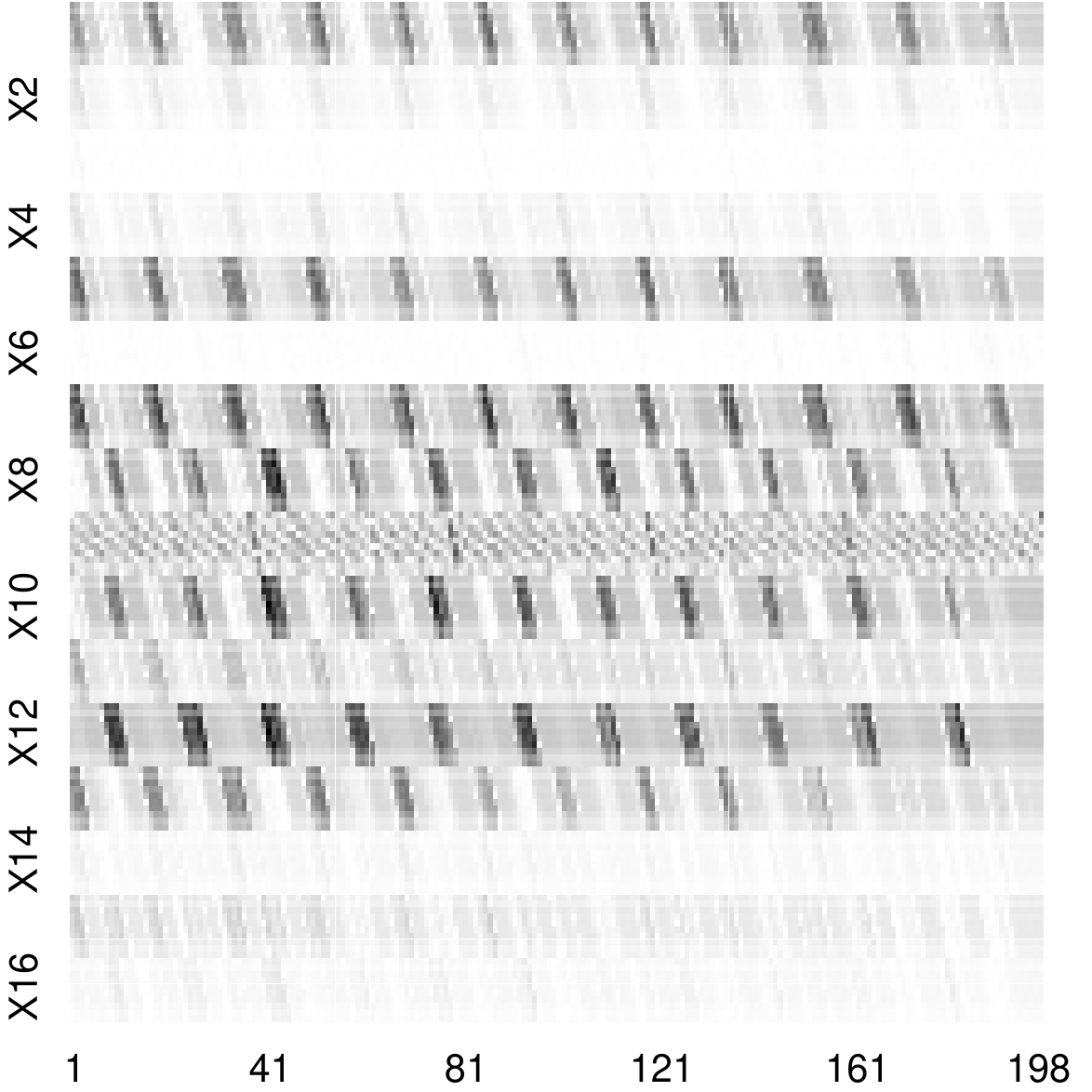}
    }\hfill
    \subfigure[Coefficient matrix from the multiscale basis.]{
        \includegraphics[width=0.45\linewidth]{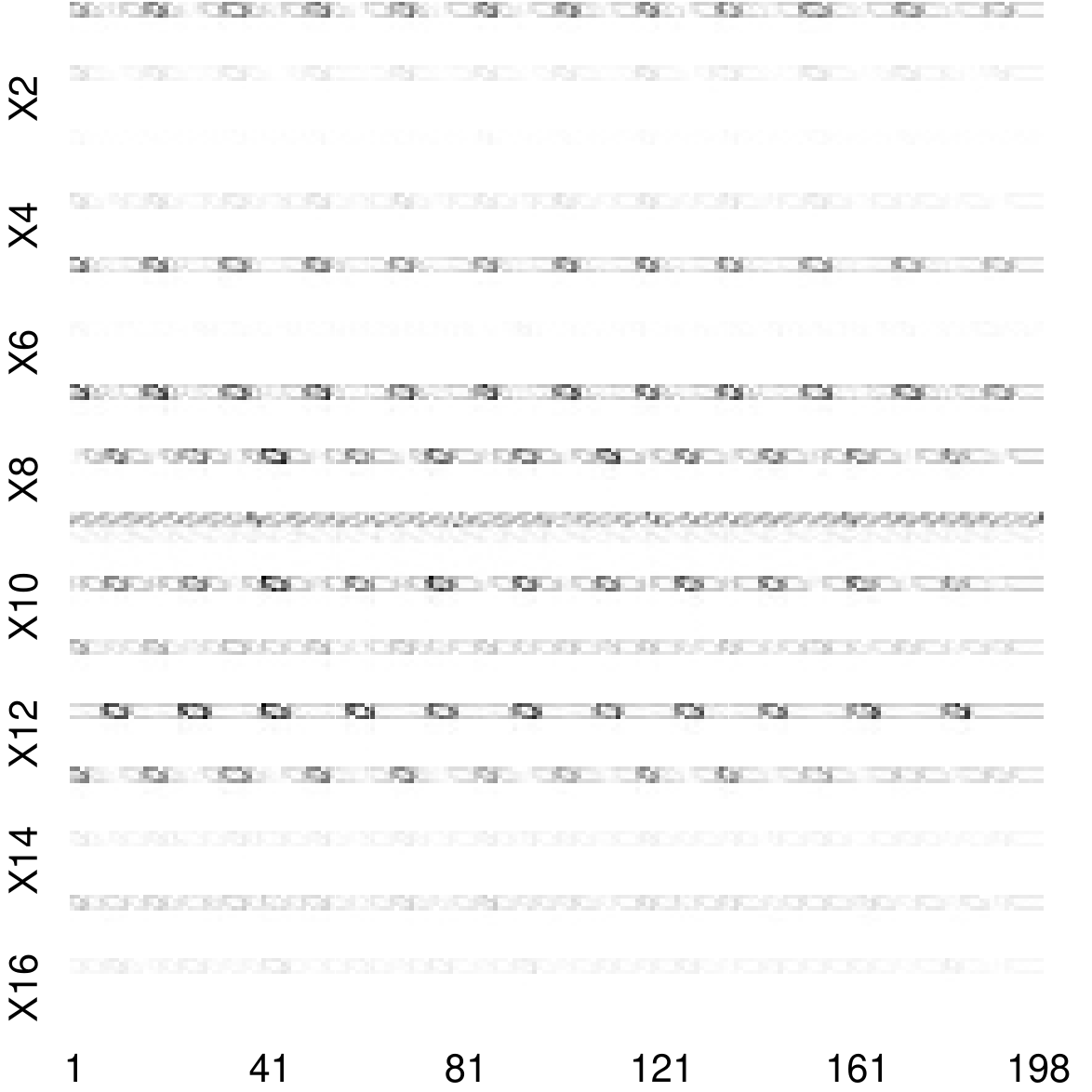}
    }\hfill
\end{figure}

Based on the explanation provided in (\ref{itm:1}) and \Cref{pps:ele_est} in \ref{app}, we introduce the following multiscale truncation strategy for \cref{mdl:FLM_GGL}. Theoretically the higher level $n$ we consider, the more precise solutions we will get, but higher level brings heavy computation. Fortunately, the following theorem implies that with the multiscale basis, relatively high levels will have insignificant effect to the coefficient matrix. This allows us to truncate the coefficient matrix, reducing computational costs while keeping coefficient matrices precise. The truncated $\tilde A^m_n$ of matrices $A_n$ for level $1\leq m\leq n$ is defined by
\begin{equation}\label{eq:trunc_strategy}
    \zkh{\tilde A^m_n}_{ij}\DEF\begin{cases}
        \zkh{A_n}_{ij},&1\leq j\leq 2^m\kh{p+1},\\
        0,&2^m\kh{p+1}<j\leq 2^n\kh{p+1}.
    \end{cases}
\end{equation}
We notice that $A_m$ is exactly the first $2^m(p+1)$ columns of $A_l$ for all $l\geq m$. To measure the difference caused by truncation, we denote by $\|A\|_2$ the $2$-norm for matrix $A$. Let $C^p[0,1]$ denote the space of all functions on $[0,1]$ possessing $p$-order continuous derivatives for $p\in\bN$.
\begin{thm}\label{thm}
    If $f\in C^p[0,1]$ and $m\in\bN$, then for $n> m$ there holds
    \[
        \normm{A_n-\tilde A^m_n}_2\leq c\cdot 2^{-mp},
    \]
    where $c>0$ is independent of $n$.
\end{thm}

\Cref{thm} implies that with multiscale basis, the truncation would not cause great impact to the coefficient matrix, as long as $m$ is sufficiently large. With such level $m$, we truncate the coefficient matrix $A_n$ and work with the sparse $\tilde A_n^m$ in computation afterwards.

\subsection{Multiscale SAFE Method}

We thoroughly present the Multiscale SAFE method with multiscale basis to discretize \cref{mdl:FLM_GGL,mdl:stage2}. It will produce sparse matrices in the discretized problems and provide a much faster way to select the sensors and estimate the kernels than SAFE method.

In MSAFE, we represent the unknown function $\gamma_k$ by a basis in the Cartesian product space $\cS\DEF\bM_n^p\otimes\bS$, where $\bM_n^p$ has a multiscale basis $\{w_j: 1\leq j\leq (p+1)2^n\}$ and $\bS$ is a cubic spline space on $[0,1]$ with basis $\{s_l: 1\leq l\leq q\}$. That is, we consider
\begin{align*}
    \gamma_k\kh{t,z}\DEF\sum_{j=1}^{\kh{p+1}2^{n}}\sum_{l=1}^{q}b_{jlk}w_j\kh{t}s_l\kh{z},\qquad 1\leq k\leq K.
\end{align*}
For $1\le i\le N$ and $1\le k\le K$, define
\begin{equation}\label{eq:trunc}
    \zkh{A_{ik}}_{jl}\DEF s_l\kh{z_i}\int_0^1\cX_{ik}\kh{\tau}w_j\kh{\tau}d\tau, \quad 1\le j\le (p+1)2^m,~1\leq l\leq q,
\end{equation}
where $\cX_{ik}\DEF X_{ik}(t_i-\delta\cdot)$ and $m\leq n$ is the selected truncation level discussed in previous paragraph. That means the matrix $A_{ik}$ is already truncated according to strategy \cref{eq:trunc_strategy}. For $1\le k\le K$, we set $[\fA_k]_{i,\cdot}=\cV(A_{ik})^\mathsf{T}$ for $1\le i\le N$ and
\begin{align*}
    \beta_k\DEF\cV(\fB_k)\quad\mbox{with}\quad[\fB_k]_{jl}\DEF b_{jlk}\quad\mbox{for}\quad1\leq j\leq (p+1)2^n,~1\leq l\leq q,
\end{align*}
where $\cV$ denotes the operator that stacks the columns of a matrix into a column vector. We then discretize the FLM \cref{mdl:FLM_GGL} with multiscale basis as
\begin{align}\label{mdl:MC_GGL}\color{blue}
    \min_{\beta_k\in\bR^{2^nq\kh{p+1}}, k\in \cK}\dkh{\normm{\sum_{k=1}^K\fA_k\beta_k-y}^2+\lambda\sum_{k=1}^K\sqrt{\beta_k^\mathsf{T} \fG_{k}\beta_k}},
\end{align}
where $\fG_k\DEF f_k\cG+\phi_{t}g_k\cG_w+\phi_{z}h_k\cG_s$, $\cG\DEF G_s\otimes G_w$, $\cG_w\DEF G_s\otimes D_w$ and $\cG_s\DEF D_s\otimes G_w$ with `$\otimes$' denoting the Kronecker product of two matrices. The gram matrices $G_w,D_w\in\bR^{(p+1)2^{n}\times (p+1)2^{n}}$ and $G_s,D_s\in\bR^{q\times q}$ are given by
\begin{align*}
    \zkh{G_w}_{jj'}\DEF\kh{w_j,w_{j'}},\quad\zkh{D_w}_{jj'}\DEF\kh{w_j'',w_{j'}''},\quad\zkh{G_s}_{ll'}\DEF\kh{s_l,s_{l'}},\quad\zkh{D_s}_{ll'}\DEF\kh{s_l'',s_{l'}''}. 
\end{align*}
Since $\fG_k$'s are symmetric positive-definite, \cref{mdl:MC_GGL} could be easily reformulated as standard group LASSO model by variable substitution. We will then use the above \cref{mdl:MC_GGL} to select the most important sensors through the multistage approach described in \Cref{sec:problem_description}.

Once the most important sensors $\cK^R$ are selected after $R$ stages, we will use the following ridge regression model to estimate the corresponding coefficient of kernels $\beta_k$, for $k\in\cK^R$:
\begin{align}\label{mdl:MC_KE}\color{blue}
    \min_{\beta_k\in\bR^{2^nq\kh{p+1}},~k\in\cK^R}\dkh{\normm{\sum_{k\in\cK^R}\fA_k\beta_k-y}^2+\sum_{k\in\cK^R}\beta_k^\mathsf{T} \fG\beta_k},
\end{align}
where $\fG\DEF \phi\cG+\phi_t\cG_w+\phi_z\cG_s$ with parameters $\phi,\phi_t,\phi_z\geq0$. To solve \cref{mdl:MC_KE}, notice that the objective function is a quadratic function, which means that \cref{mdl:MC_KE} could be easily solved as a linear system.

\section{Numerical Experiments}

We will implement the proposed MSAFE method on the real data sets \cite{Stallrich2020} consisting of EMG and movement data from an AB subject's right limb. We will also compare its performance with the original SAFE method proposed in \cite{Stallrich2020}. Numerical results demonstrate that the proposed MSAFE method could achieve better or similar performance in sensor selection and prediction error with a significantly less computation time than SAFE. All experiments in this section are executed on \code{R} \cite{R}, within Windows 10 on an Intel Core i7 CPU @ 3.60 GHz and 16GB RAM. 

\subsection{Data Description and Preprocessing}
We first briefly describe the data sets. We consider two different patterns of movements, constant and random, and each of them contains 3 different movement data of finger and wrist, respectively. These give us 12 different data sets in total. There are 15 EMG sensors placed on the subject's limb. An external EMG signal unrelated to movement is also added to address the validity of sensor selection. Therefore in each data set, we will have 16 EMG signals $\{X_k(t):t\in T\}_{k=1}^{16}$ at $198$ different sampling time $T\coloneqq\{t_i\}_{i=1}^{198}$, where $X_9$ is the unrelated one. Moreover, the displacement of finger flexion/extension or wrist flexion/extension $\{z(t):t\in T\}$ at the corresponding times $T$ are also collected.

We next describe the preprocessing of the raw data of EMG signals and displacement in the convenience of numerical implementation. The displacement data $\{z_i\DEF z(t_i)\}_{i=1}^{198}$ and historical EMG signals $S_{ik}\DEF\{X_k(t):t\in[t_i-\delta,t_i]\cap T\}$ with window size $\delta=1/3$ for $1\leq k\leq16$ and $1\leq i\leq198$ are extracted from the raw data at 198 sampling time $\{t_i\}_{i=1}^{198}\subset T$. To get the corresponding movement velocity $\{z^\prime(t_i)\}_{i=1}^{198}$, six-order spline basis with third-order regularization are used to get a fit $\hat z(t)$ out of data $\{z(t):t\in T\}$, then $\{y_i\DEF\hat z'(t_i)\}_{i=1}^{198}$ could be computed explicitly. As for the integral \cref{mdl:FLM_ORI}, MSAFE uses continuous piecewise linear functions to interpolate the discretely sampled EMG data $\{X_k(t):t\in[t_i-\delta,t_i]\cap T\}$, and gets approximations of the continuous signals $\{\cX_{ik}(t):t\in[0,1]\}$ for $1\leq k\leq16$ and $1\leq i\leq198$ with explicit formulas. The integral \cref{mdl:FLM_ORI} and the matrices $\mathsf{A}_k$ for $1\leq k\leq 16$ in \cref{mdl:MC_GGL,mdl:MC_KE} can then be approximated.

\subsection{Experiment Setups}

In MSAFE, we will represent the kernels $\gamma_k(t,z)$ for $1\leq k\leq16$ in the space $\cS=\bM_2^3\otimes \bS_{10}$, where $\bM_2^3$ denotes the multiscale piecewise cubic polynomial space of level $2$ (see more details in \ref{app}) and $\bS_{10}$ is cubic spline space with dimension 10. We point out that the resulting space $\cS$ for MSAFE is of dimension $160$, which is larger than the one of SAFE method. Actually SAFE method adopts the space $\bS_{10}\otimes\bS_{10}$.

We start with the sensor selection. The tuning parameters $\lambda,\phi_t,\phi_z$ in the sensor selection \cref{mdl:MC_GGL} are set in such a way that $\log\lambda$ takes values from $-20$ to $0$ with step $0.25$ and $\log\phi_{t},\log\phi_{z}$ ranges from $-10$ to $0$ with step $2.5$. We will use 5-fold cross validation to select them in each of the following stages. The sequentially updated parameters $\{f_k,g_k,h_k\}_{k=1}^{16}$ in \cref{mdl:MC_GGL} are initialized to be $1$. We use the \code{R} package \code{gglasso} to solve \cref{mdl:MC_GGL} with the initial values at the first stage. We then get an active variable set $\cK^1$ and update the values of $\{f_k,g_k,h_k\}_{k=1}^{16}$ according to \cref{mdl:FLM_GGL_fgh}.  The second stage model \cref{mdl:MC_GGL} will be solved with the updated values of $\{f^1_k,g^1_k,h^1_k\}_{k=1}^{16}$ and the active variable set $\cK^1$. We repeat this process for 5 stages and obtain the final active variable set $\cK^5$ with both methods.

We next continue with the kernel estimation. Specifically, we will solve \cref{mdl:MC_KE} with the active variable set $\cK^5$.  
For SAFE method, $\phi=0$ as claimed in \cite{Stallrich2020} and for MSAFE, we set candidates $\phi$ such that $\log_{10}\phi\in\{-1,-2,-3,-4,-5\}$ for cross validation. In both methods, candidates of $\phi_t$, $\phi_z$ for cross validation and other setups are kept identical to the ones in sensor selection part, respectively.

We remark that the SAFE method in \cite{Stallrich2020} represents the kernels $\gamma_k$ in $\bS_{10}\otimes \bS_{10}$. We already showed in \Cref{fig:hist_2} that the coefficient matrices generated by the multiscale basis are much more sparse than those generated by the spline basis. In other words, the coefficient matrices in the proposed MSAFE method have a majority of entries close to 0. Furthermore, \Cref{fig:hist_2} indicates that the additional sparsity that is not modeled by truncation strategy \cref{eq:trunc_strategy}, and the spline term in \cref{eq:trunc} could contribute to some extra sparsity as well. In regard of these, we further truncate those small entries in $\mA_k$'s, only keep 10\% of entries to be nonzero in those coefficient matrices $\mA_k$'s.

During this study, the latest version of \code{gglasso} package does not take sparse matrices. To have a fair comparison with SAFE and fully demonstrate advantages of multiscale basis, we delve into the \code{Fortran} codes of the package, implement the sparse matrices multiplication (see e.g.\ \cite{saad1990sparskit}) and incorporate it into functions from the package \code{gglasso}.

\subsection{Experiment Results}

We will compare the performance of the proposed MSAFE method with original SAFE in the following aspects: the selected sensors, the prediction error with the estimated kernels, and the total computational time in sensor selection and kernel estimation. More precisely, the prediction error for a specified method and data set is defined by
\[
    {\rm MSE}=\frac{1}{5}\sum_{i=1}^5\sum_{j\in F_i}\frac{\kh{\hat y_j-y_j}^2}{|F_i|},
\]
where $F_i$ identifies the $i$-th test fold of the 5-fold cross validation, $\{y_j\}_{j\in F_i}$ are the actual responses for testing, and $\{\hat y_j\}_{j\in F_i}$ are the predicted values of the kernels estimated on corresponding training folds.

\begin{table}[!t]
    \centering
    \caption{Performance metrics for sensor selection at the final stage for constant (top three rows) and random (bottom three rows) finger and wrist movement patterns, and the CV MSE means and time costs for each of the data sets.}
    \label{tab:hand_mse}\small
    \begin{tabular}{c|c|c|S|
        S[table-format = 3.2,scientific-notation=false,round-mode=places,round-precision=2,table-parse-only=false]
        |c|S|
        S[table-format = 4.2,scientific-notation=false,round-mode=places,round-precision=2,table-parse-only=false]}
        \multirow{3}{*}{Data set} & \multirow{3}{*}{Method} & \multicolumn{3}{c|}{Finger Movement} & \multicolumn{3}{c}{Wrist Movement} \\\cline{3-8}
        && \mlc{Selected\\Sensor}& {CV MSE}& {Time (min)}& \mlc{Selected\\Sensor}& {CV MSE}& {Time (min)}\\\hline
        \multirow{2}{*}{Const \#1}
            & SAFE  & 7, 12 & 0.07793521 & 504.769  & 2, 11, 15 & 0.03930965 & 527.5565\\
            & MSAFE & 7, 12 & 0.07161073 & 47.15217 & 8, 15     & 0.04388828 & 68.94117 \\\hline
        \multirow{2}{*}{Const \#2}
            & SAFE  & 7, 12 & 0.09937427 & 525.9193 & 11, 15 & 0.04807177 & 734.1775\\
            & MSAFE & 7, 12 & 0.08298230 & 43.7185  & 11, 15 & 0.04856872 & 88.16233 \\\hline
        \multirow{2}{*}{Const \#3}
            & SAFE  & 7, 12 & 0.101227   & 669.3632 & 2, 9, 11, 15 & 0.05065761 & 886.0683\\
            & MSAFE & 7, 12 & 0.09279451 & 76.37567 & 11, 15       & 0.06015498 & 113.0318\\\hline
        \multirow{2}{*}{Rand \#1}
            & SAFE  & 7, 12 & 0.2020697  & 358.5368 & 8, 11, 15 & 0.1135957  & 699.4698\\
            & MSAFE & 7, 12 & 0.20479668 & 54.80233 & 8, 15     & 0.11809699 & 120.0632\\\hline
        \multirow{2}{*}{Rand \#2}
            & SAFE  & 5, 7, 12 & 0.1724732  & 593.2337 & 2, 8, 11, 15 & 0.1094774  & 1345.794\\
            & MSAFE & 5, 7, 12 & 0.14343446 & 103.1722 & 11, 15       & 0.12035256 & 201.8712\\\hline
        \multirow{2}{*}{Rand \#3}
            & SAFE  & 7, 12 & 0.1749214  & 725.365  & 11, 15 & 0.08222036 & 615.2748\\
            & MSAFE & 7, 12 & 0.17487612 & 122.4722 & 11, 15 & 0.08472709 & 65.80967
    \end{tabular}
\end{table}

Anatomy of hand movements provides us the most important sensors' placement related to the finger/wrist flexion/extension movements. Let $\cK\DEF\cK_F\cup\cK_E$ denote the index set of those most relevant sensors to the movement of interest, where $\cK_F$ and $\cK_E$ split $\cK$ into groups corresponding to flexion and extension. It was claimed in \cite{Stallrich2020} that for finger movement, $\cK_F=\{12\}$ and $\cK_E=\{5,7\}$; for wrist movement, $\cK_F=\{8,10,11,14\}$ and $\cK_E=\{2,7,13,15\}$. Notice that sensors $\{1,3,4,6,9,16\}$ are irrelevant to the movements of interest.

\Cref{tab:hand_mse} shows the performance of SAFE and MSAFE in sensor selection, cross validation mean square error, and the computational time for each data set. It is direct to observe that the proposed MSAFE method selects the same sensors as SAFE in most of the data sets. For finger movement data sets, both methods select exactly the same sensors. For wrist movement data sets, MSAFE method tends to select fewer but more important sensors. For example, for Constant \#3 data of wrist movement, sensor 9 is incorrectly chosen by SAFE and is successfully filtered by MSAFE.

Moreover, the cross validation mean square errors of MSAFE are almost the same as those of SAFE in every data set. However, as \Cref{fig:REAL} reveals, the overall computational time of MSAFE is remarkably less than that of SAFE in every data set; MSAFE costs only about 10\%$\sim$15\% time of that in the SAFE method. These results confirm that the multiscale polynomial basis used in MSAFE have brought a huge advantage in computational cost, while maintaining the prediction accuracy.

\begin{figure}[!t]
    \centering
    \caption{CV MSE means and standard deviations for the optimal tuning parameters of the last selection stage (top panel) and the time cost (bottom panel). MSAFE has comparable prediction error than SAFE, while time plots illustrate the advanced efficiency brought by multiscale basis.}\label{fig:REAL}
    \hfill\subfigure{
        \includegraphics[width=0.45\linewidth]{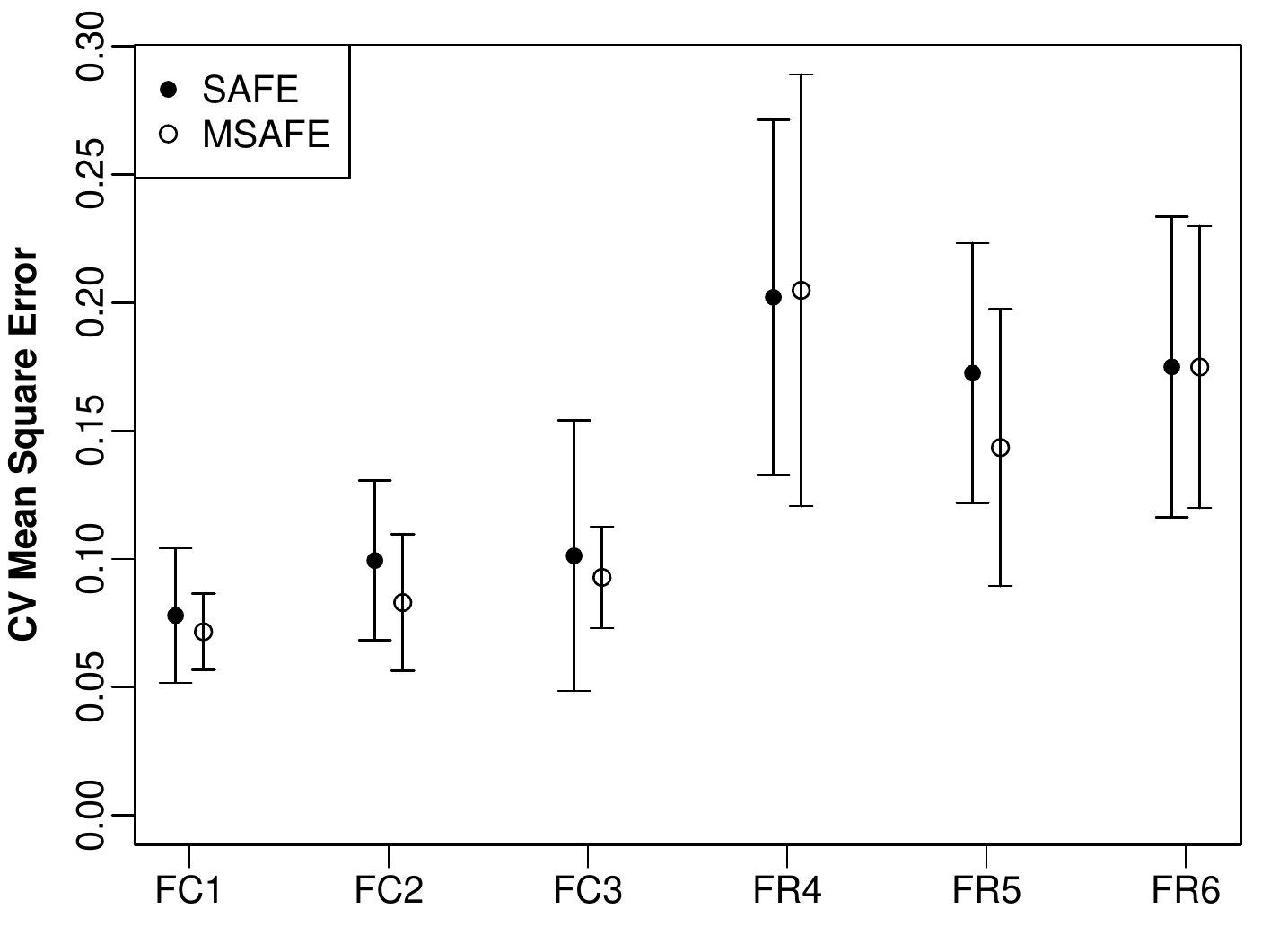}
    }\hfill
    \subfigure{
        \includegraphics[width=0.45\linewidth]{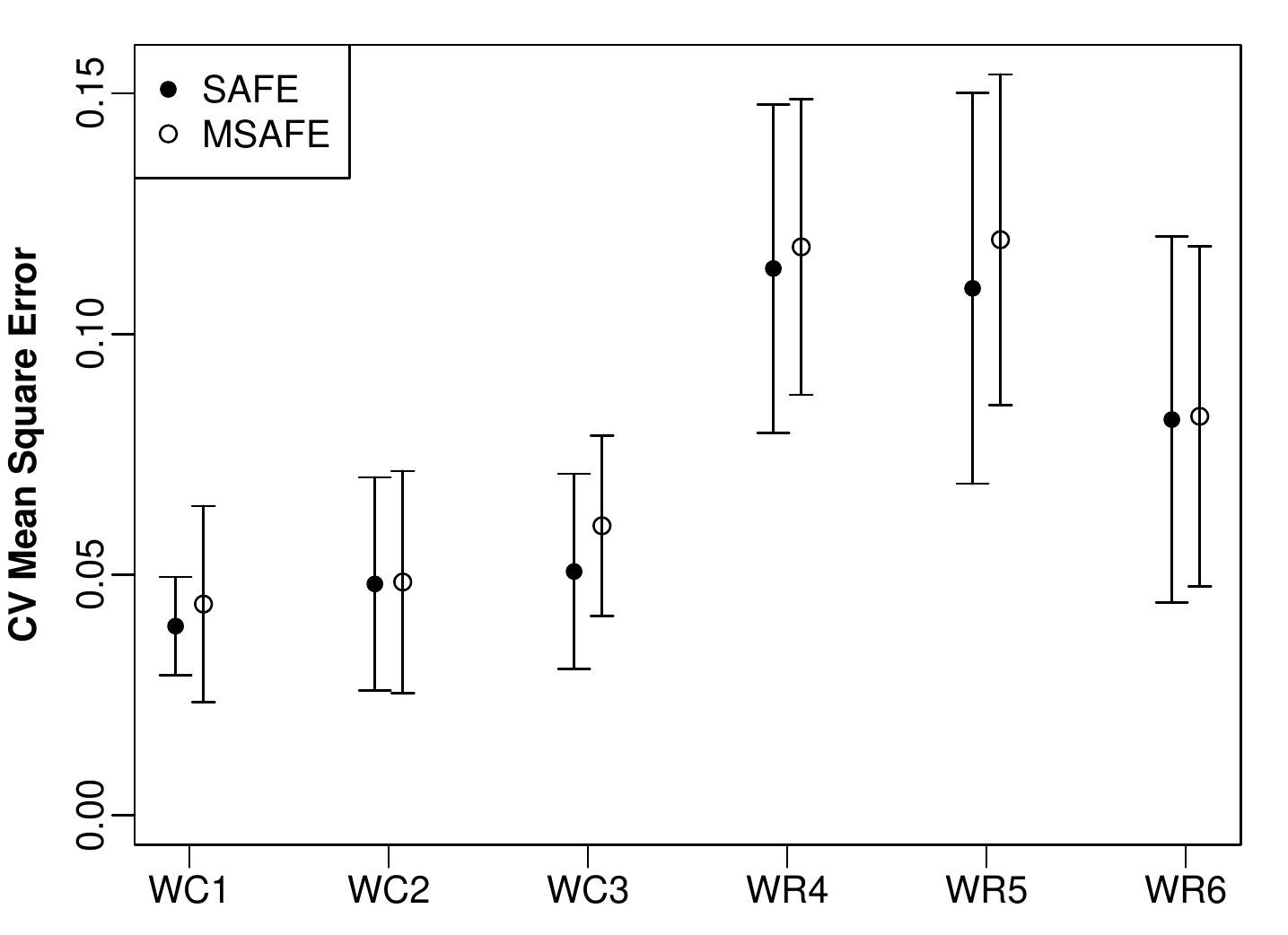}
    }\hfill\\
    \hfill\subfigure{
        \includegraphics[width=0.45\linewidth]{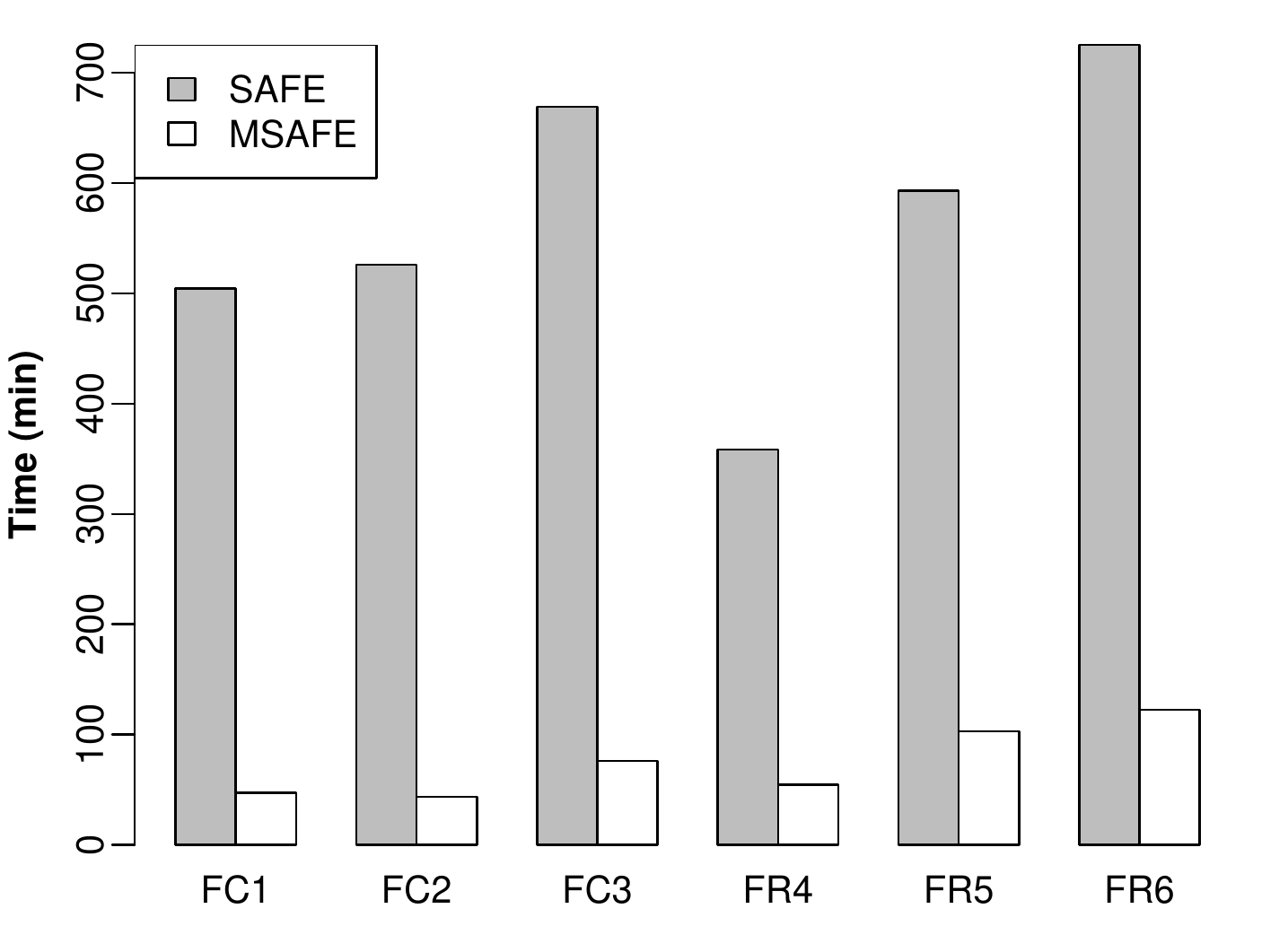}
    }\hfill
    \subfigure{
        \includegraphics[width=0.45\linewidth]{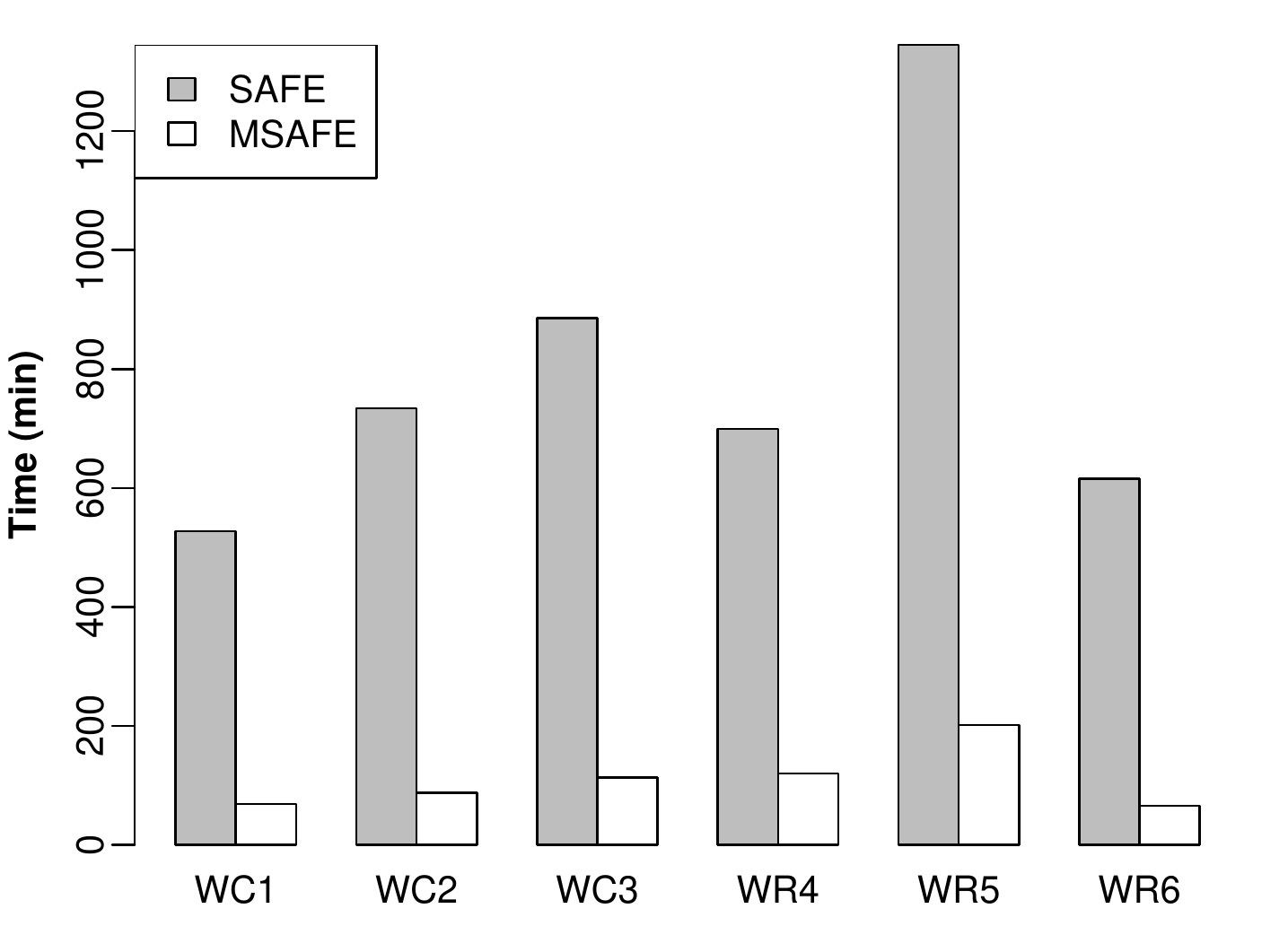}
    }\hfill
\end{figure}

\section{Simulation Study}

This section tests the robustness of SAFE and MSAFE in sensor selection against the impact of covariance misspecification, based on simulated data with correlated noises studied in \cite{Stallrich2020}. We use the data set FC3 and the corresponding estimated kernels $\hat\gamma_7$ and $\hat\gamma_{12}$ from SAFE method to generate the responses. Specially, we generate data by
\begin{equation}\label{eq:simu_gen}
    y_i=\sum_{k=7,12}\int_0^1X_k\kh{t_i-\delta \tau}\hat\gamma_k\kh{\tau,z_i}d\tau+\varepsilon_i,\qquad 1\leq i\leq N,
\end{equation}
where $X_k$'s and $t_i$'s are from data set FC3 in \cite{Stallrich2020}, $\hat\gamma_7$ and $\hat\gamma_{12}$ are estimated kernels of FC3 by SAFE method, $\{\varepsilon_i\}_{i=1}^N$ are zero-mean multivariate Gaussian noises with covariance matrix $\Sigma\in\bR^{N\times N}$ such that
\[
    \cov\zkh{\varepsilon_i,\varepsilon_j}=\sigma_h^2\zkh{\delta_{ij}+\theta\exp\kh{-\kh{i-j}^2/\eta^{2}}}.
\]
Here $\theta>0$ is related to the dominant sources of dependence, $\eta>0$ controls the correlation decay where larger values imply slower correlation decay, and $\sigma_h>0$ is chosen such that $\Sigma_{ii}=\sigma_h^2(1+\theta)$ equals the standard deviations of CV MSE means from the SAFE method on FC3 data set. The simulation experiment runs over factors $\theta\in\{0.25,10,100\}$ and $\eta\in\{10,100\}$, and $J=100$ individual sets of noises on each scenario. We use the same setting as \cite{Stallrich2020} and will compare the performance of MSAFE with SAFE. We set the number of stages $R=2$ for both methods.

For the generated $J$ different sets of noises, we will report the following quantities in \Cref{tab:simulation_SAFE}. 
\begin{itemize}
    \setlength\itemsep{0ex}
    \item   \textbf{Mean Size}: ${\sum}{}_{j=1}^{J}|\cK^2_j|/J$, where $\cK^2_j$ is the set of selected sensors after 2 stages for each $1\leq j\leq J$. We note that the ideal value is $2$ with $\{7,12\}$ as the ground-truth sensors for all data sets.
    \item   \textbf{Mean False Positive}: ${\sum}{}_{j=1}^{J}|\cK^2_j\setminus\{5,7,12\}|/J$. The ideal value is $0$.
    \item   \textbf{Mean CV MSE}: ${\sum}{}_{j=1}^{J}\text{MSE}_j/J$, where $\text{MSE}_j$ is the CV MSE for $j$-th data set.
    \item   \textbf{Mean Time}: ${\sum}{}_{j=1}^{J}T_j/J$, where $T_j$ is the time cost for $j$-th data set.
\end{itemize}
We also display in \Cref{fig:simulation_SAFE} the mean and standard deviation of the CV MSE of both SAFE and MSAFE methods with each scenario of $\theta$ and $\eta$.

We observe that both methods select the correct sensors 7 and 12 on all data sets. However in every setting of $\theta$ and $\eta$, MSAFE selects fewer sensors and less incorrect selections than SAFE. MSAFE method merely selects 45\%$\sim$65\% extra sensors as SAFE does; specially MSAFE selects only 8\%$\sim$23\% misspecified sensors than SAFE does. Moreover, MSAFE tends to have less prediction error in every scenario; specially, MSAFE has 70\%$\sim$80\% mean square error as the one of SAFE. Overall, MSAFE method is more robust against the covariance misspecification than SAFE. Finally, the computational time of MSAFE on those simulation data sets is always about 10\%$\sim$18\% of that of the SAFE method, which once again corroborates the stability and speed advantage of multiscale piecewise polynomial basis.

\begin{table}[t!]
    \centering\small
    \caption{Performance metrics across $100$ data sets for various covariance settings with simulation model \cref{eq:simu_gen}.}
    \label{tab:simulation_SAFE}
    \begin{tabular}{c|c|
    S[table-format = 1.2,scientific-notation=false,round-mode=places,round-precision=2,table-parse-only=false]|S[table-format = 1.2,scientific-notation=false,round-mode=places,round-precision=2,table-parse-only=false]|S[table-format = 1.2,scientific-notation=false,round-mode=places,round-precision=2,table-parse-only=false]|S[table-format = 1.2,scientific-notation=false,round-mode=places,round-precision=2,table-parse-only=false]|S|S|S[table-format = 3.2,scientific-notation=false,round-mode=places,round-precision=2,table-parse-only=false]|S[table-format = 2.2,scientific-notation=false,round-mode=places,round-precision=2,table-parse-only=false]}
        \multirow{2}{*}{$\theta$} & \multirow{2}{*}{$\eta$} & \multicolumn{2}{c|}{Mean Size}&\multicolumn{2}{c|}{Mean False Positive}&\multicolumn{2}{c|}{Mean CV MSE}&\multicolumn{2}{c}{Mean Time (min)}\\\cline{3-10}
        &&{SAFE}&{MSAFE}&{SAFE}&{MSAFE}&{SAFE}&{MSAFE}&{SAFE}&{MSAFE}\\\hline
        
\multirow{2}{*}{0.25}  & 10 & 3.61224489795918 & 2.35714285714286 & 1.59183673469388 & 0.357142857142857 & 0.00428712022360463 & 0.00334442398269391 & 164.141408163265 & 30.6631836734694\\
  & 100 & 3.89795918367347 & 2.41836734693878 & 1.86734693877551 & 0.418367346938776 & 0.00459344774833186 & 0.00363880089772183 & 169.819350340136 & 25.3407159863946\\
\hline\multirow{2}{*}{10}  & 10 & 3.60204081632653 & 2.13265306122449 & 1.55102040816327 & 0.13265306122449 & 0.00407762859154969 & 0.0029511640684577 & 163.064540816327 & 25.9250918367347\\
  & 100 & 4.85714285714286 & 2.37755102040816 & 2.72448979591837 & 0.377551020408163 & 0.00447410704166156 & 0.00354767632414209 & 189.712003401361 & 22.4634455782313\\
\hline\multirow{2}{*}{100}  & 10 & 3.75510204081633 & 2.20408163265306 & 1.68367346938776 & 0.204081632653061 & 0.00404043806799203 & 0.00283672644220101 & 168.376166666667 & 25.8468928571429\\
  & 100 & 5.19387755102041 & 2.38775510204082 & 2.97959183673469 & 0.387755102040816 & 0.00446734112151894 & 0.00348873383996535 & 200.888066326531 & 21.4969914965986
    \end{tabular}
\end{table}

\begin{figure}[t!]
    \centering
    \caption{Plots of CV MSE and time of SAFE and MSAFE methods across $100$ data sets for various covariance settings with simulation model \cref{eq:simu_gen}. The $x$-axis labels denote the pair $(\theta,\eta)$.}\label{fig:simulation_SAFE}
    \hfill\subfigure[CV MSE on simulated data.]{
        \includegraphics[width=0.45\linewidth]{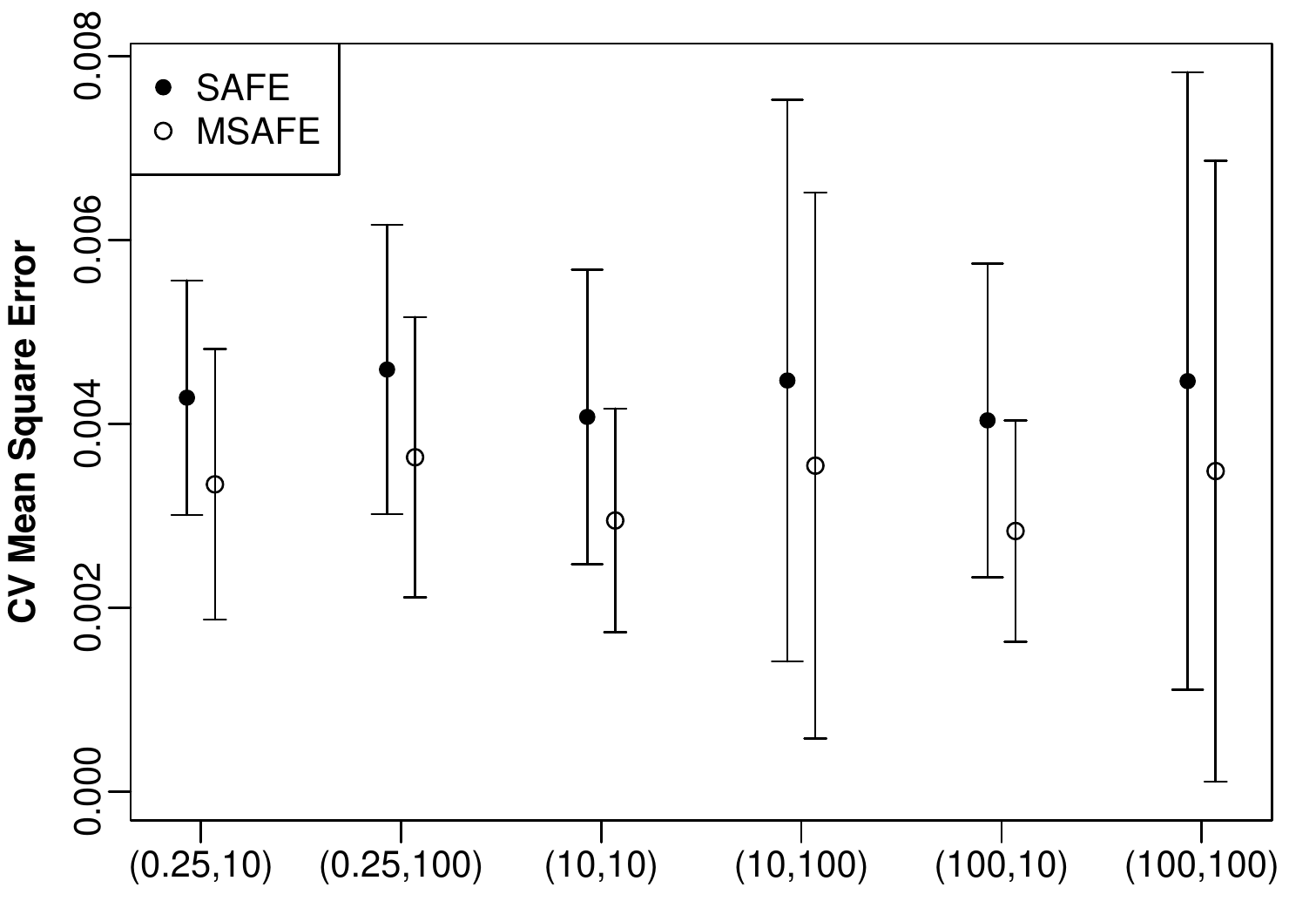}
    }\hfill
    \subfigure[Time on simulated data.]{
        \includegraphics[width=0.45\linewidth]{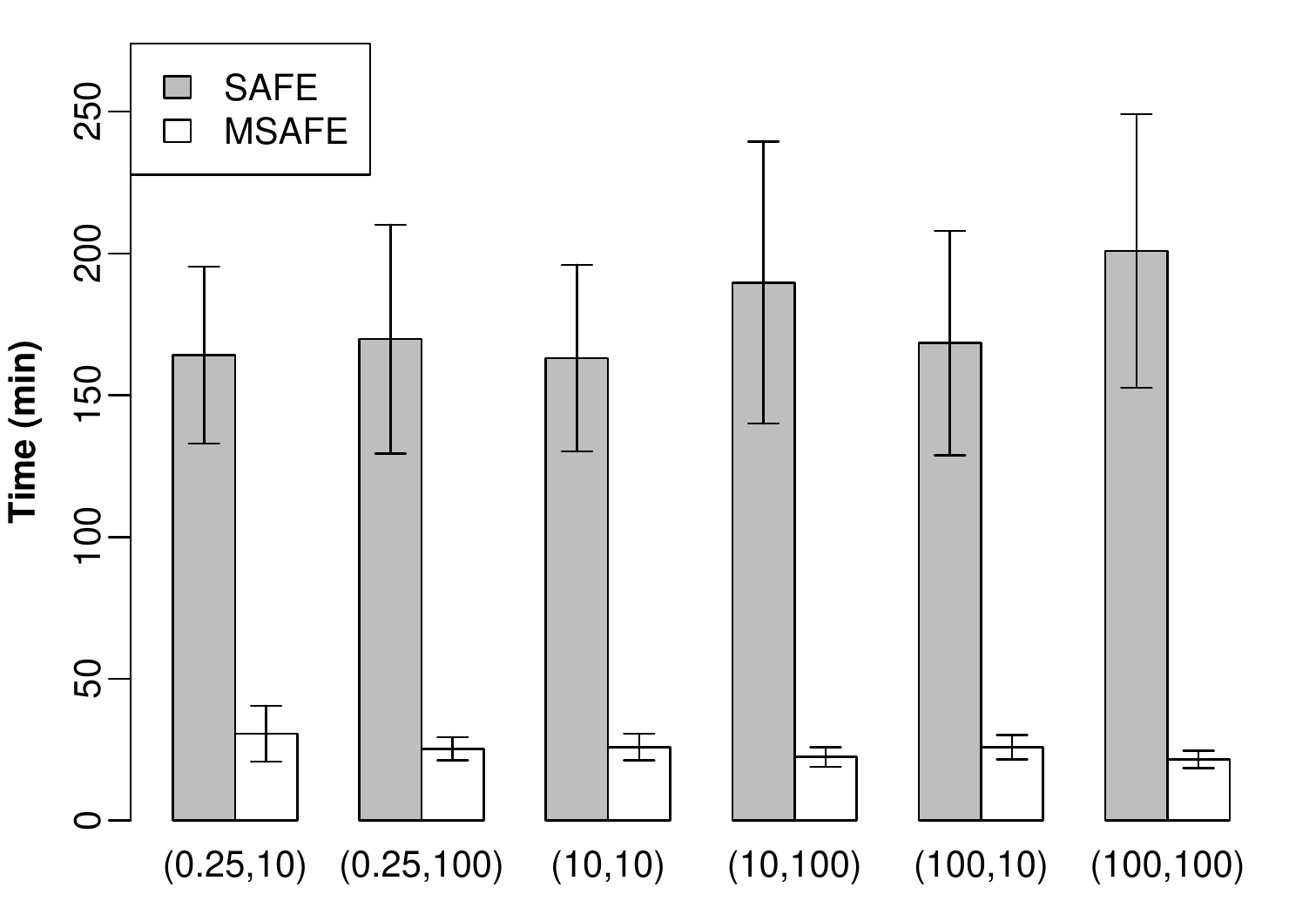}
    }\hfill
\end{figure}

\section{Conclusions}

To perform fast and precise algorithm for robotic prosthesis controllers, we propose MSAFE method based on SAFE, with the multiscale piecewise polynomial basis to discretize the integral operator in FLM. Multiscale basis systematically generates sparse coefficient matrices, accelerating the calculation and improving the stability of original SAFE method. Compared to single-scale spline basis SAFE method, MSAFE with multiscale basis costs only 10\%$\sim$15\% computational time on the hand movement data, while performing better sensor selection and comparable prediction accuracy. We also test the robustness of sensor selection for multi- and single-scale basis against correlation noise. Experiments on simulated data shows that with various patterns of correlated noise, MSAFE always has slighter misspecification and prediction error than SAFE. Both real-data experiments and simulation studies corroborate the efficiency and stability of the proposed MSAFE method.

\section*{Acknowledgement}

The authors are grateful to Professors Ana-Maria Staicu and Jonathan Stallrich of North Carolina State University for their generosity in providing the hand/wrist movement data and fruitful discussions.

\appendix
\section{Multiscale Piecewise Polynomial Basis}\label{app}

In this appendix, we present the multiscale piecewise cubic polynomial basis on $\Omega\coloneqq[0,1]$ and its properties. For $n\geq 0$, we denote by $\bM_n^p$ the linear space of all piecewise polynomials of degree less than or equal to $p$, supported on $\Omega$ with nodes $\{i/2^n\}_{i=0}^{2^n}$. Abbreviation $\bM_n$ is used when the degree $p$ is clear from the context. By definition, we have the nestedness $\bM_{n}\subset\bM_{n+1}$. This enables us to define wavelet subspace $\bW_{n+1}\subset\bM_{n+1}$ such that $\bW_{n+1}\perp\bM_n$ in $L^2(\Omega)$ sense, and $\bM_{n+1}=\bM_n\oplus^\perp\bW_{n+1}$, where `$\oplus^\perp$' denotes the direct sum of two perpendicular spaces. With these spaces, the multiscale decomposition of function space $\bM_{n}$ could be represented by
\begin{equation}\label{eq:multi_dec}
    \bM_{n}=\bM_0\oplus^\perp\bW_1\oplus^\perp\bW_2\oplus^\perp\cdots\oplus^\perp\bW_n.
\end{equation}
Such a decomposition has a spectacular property that $\bW_n$ can be constructed based on $\bW_1$. To see this, define $\Phi\coloneqq\{\phi_0,\phi_1\}$ where $\phi_0\coloneqq\cdot/2$, $\phi_1\coloneqq(1+\cdot)/2$, and transformations of functions $f\in L^\infty(\Omega)$ such that
\[
    \cT_0 f\coloneqq f\circ\phi_0^{-1},\qquad \cT_1 f\coloneqq f\circ\phi_1^{-1}.
\]
One can learn from \cite{micchelli1994using} that
\[
    \bW_{n+1}=\cT_0\bW_n\oplus^\perp\cT_1\bW_n,\qquad\text{for all }n\in\bN.
\]
Moreover, if $W_n$ is a basis of $\bW_{n}$, then 
\[
    W_{n+1}\DEF\dkh{\cT_iw:w\in W_{n},i\in\dkh{0,1}}
\]
is a basis of $\bW_{n+1}$. This result reveals the relation between wavelet space $\bW_n$ and $\bW_1$, and a systematic way to generate basis $W_n$ of $\bW_n$ out from basis $W_1$ of $\bW_1$. \[
    \cT_e\DEF\cT_{e_1}\circ\cT_{e_2}\circ\cdots\circ\cT_{e_n}.
\]
Then for $n>1$ we have that
\begin{equation}\label{eq:multi_gen}
    \bW_n=\bigcup_{e\in\dkh{0,1}^{n-1}}\cT_e\bW_1,\qquad W_n=\bigcup_{e\in\dkh{0,1}^{n-1}}\cT_e W_1,
\end{equation}
and we would have $W_n$ as a basis of $\bW_n$.

The relation above not only serves as a systematical generator of basis functions for high levels, but also leads us to the estimation of coefficient matrix entries. The proposition following claims that as level $n$ increases, the entries of coefficient matrix $\mathsf{A}$ in \cref{eq:A} will decay exponentially. Define $C^p[0,1]$ as the space of all functions on $[0,1]$ possessing $p$-order continuous derivatives for $p\in\bN$, and for $f\in C[0,1]$, define $\|f\|_\infty\DEF\max_{t\in[0,1]}|f(t)|$.


\begin{pps}\label{pps:ele_est}
    Suppose that $\{W_n\}_{n=1}^\infty$ is a sequence of multiscale piecewise polynomial bases of degree $p\geq 0$ generated by \cref{eq:multi_gen}. If $f\in C^p[0,1]$, then for any $w\in W_n$ with $n\geq 1$, there holds
    \begin{equation}\label{eq:ele_upbound}
        \norm{\int_0^1f\kh{\tau}w\kh{\tau}d\tau}\leq c_p\cdot2^{-\kh{p+1}n}\normm{f^{\kh{p}}}_\infty,
    \end{equation}
    with positive constant
    \[
        c_p\DEF\frac{\sqrt{2p+3}}{\kh{p+1}!2^{p+1}}\max_{v\in W_1}\normm{v}_{L^2[0,1]}.
    \]
\end{pps}
\begin{proof}
    By \cref{eq:multi_gen}, for any $w\in W_{n+1}$ with $n\geq 0$ we have $w=v\kh{2^n\cdot-i}$ for some $v\in W_1$ and $0\leq i\leq2^n-1$. Then
    \begin{align*}
        \int_0^1f\kh{\tau}w\kh{\tau}d\tau&=\int_0^1f\kh{\tau}v\kh{2^n\tau-i}d\tau=2^{-n}\int_0^1f\kh{2^{-n}\kh{\tau+i}}v\kh{\tau}d\tau.
    \end{align*}
    On the other hand, notice that by Taylor expansion, we have
    \[
        f\kh{2^{-n}\kh{\tau+i}}=\sum_{l=0}^{p-1}\frac{f^{\kh{l}}\kh{i2^{-n}}}{l!2^{nl}}\tau^l+\int_0^\tau\frac{t^p}{p!2^{pn}}f^{\kh{p}}\kh{2^{-n}\kh{t+i}}dt.
    \]
    Therefore we have
    \begin{align*}
        \norm{\int_0^1f\kh{\tau}w\kh{\tau}d\tau}&=\frac{1}{p!2^{\kh{p+1}n}}\norm{\int_0^1\!\!\!\int_0^{\tau} t^pf^{\kh{p}}\kh{2^{-n}\kh{t+i}}v\kh{\tau}dtd\tau}\\
        &\leq\frac{\normm{f^{\kh{p}}}_\infty}{p!2^{\kh{p+1}n}}\norm{\int_0^1\!\!\!\int_0^\tau t^pv\kh{\tau}dtd\tau}\\
        &=\frac{\normm{f^{\kh{p}}}_\infty}{\kh{p+1}!2^{\kh{p+1}n}}\norm{\int_0^1\tau^{p+1}v\kh{\tau}d\tau}.
    \end{align*}
    Next we estimate the last integral. Notice the fact that monic Legendre polynomials $L_{p+1}$ of degree $p+1$ has the smallest $L^2$-norm among monic polynomials with the same degree on $[-1,1]$. Then since $v\in W_{1}$ has vanishing moment $p+1$, we have
    \begin{align*}
        \norm{\int_0^1\tau^{p+1}v\kh{\tau}d\tau}&=\frac{1}{2^{p+2}}\norm{\int_{-1}^1\kh{t+1}^{p+1}v\kh{\frac{t+1}{2}}dt}\\
        &=\frac{1}{2^{p+2}}\norm{\int_{-1}^1L_{p+1}\kh{t}v\kh{\frac{t+1}{2}}dt}.
    \end{align*}
    Then by Cauchy inequality,
    \begin{align*}
        \norm{\int_0^1\tau^{p+1}v\kh{\tau}d\tau}&\leq \frac{\sqrt{2}}{2^{p+2}}\normm{L_{p+1}}_{L^2\zkh{-1,1}}\normm{v}_{L^2\zkh{0,1}}\\
        &=\frac{\sqrt{2}}{2^{p+2}}\frac{2^{p+1}\kh{\kh{p+1}!}^2}{\kh{2p+2}!}\sqrt{\frac{2}{2p+3}}\normm{v}_{L^2\zkh{0,1}}\\
        &=\frac{\normm{v}_{L^2\zkh{0,1}}}{{\binom{2p+2}{p+1}}\sqrt{2p+3}}\leq\frac{\sqrt{2p+3}}{2^{2\kh{p+1}}}\normm{v}_{L^2\zkh{0,1}},
    \end{align*}
    where the $L^2$-norm of $L_{p+1}$ and the estimation of central binomial coefficient $\binom{2n}{n}\geq 4^n/(2n+1)$ are applied. Then combining the estimations above gives the final bound
    \[
        \norm{\int_0^1f\kh{\tau}w\kh{\tau}d\tau}\leq \frac{\sqrt{2p+3}}{\kh{p+1}!2^{\kh{p+1}\kh{n+2}}}\normm{f^{\kh{p}}}_\infty\max_{v\in W_1}\normm{v}_{L^2\zkh{0,1}},
    \]
    which is exactly the desired.
\end{proof}

Based on \Cref{pps:ele_est} now we are capable to prove \Cref{thm}, which could be used to determine the multiscale level needed for a certain accuracy level. Recall that the multiscale coefficient matrix for $n$ level is $A_n$, and the truncated $\tilde A^m_n$ of matrices $A_n$ for level $1\leq m\leq n$ is
\[
    \zkh{\tilde A^m_n}_{ij}\DEF\begin{cases}
        \zkh{A_n}_{ij},&1\leq j\leq 2^m\kh{p+1},\\
        0,&2^m\kh{p+1}<j\leq 2^n\kh{p+1}.
    \end{cases}
\]
Notice that for $n\geq m$, $A_m$ is exactly the first $2^m(p+1)$ columns of $\tilde A^m_n$. For matrix $A$, define $\|A\|_2$ and $\|A\|_F$ as the matrix $2$- and Frobenius-norm of $A$ respectively.
\begin{proof}[Proof of \Cref{thm}]
    By \Cref{pps:ele_est}, for all $n>m$ we have
    \begin{align*}
        \normm{A_n-\tilde A^m_n}_F^2&\leq c_p^2\normm{f^{\kh{p}}}_{\infty}^2\kh{p+1} N\sum_{i=m+1}^{n}2^{i-1}2^{-2\kh{p+1}i}\\
        &\leq c_p^2\normm{f^{\kh{p}}}_{\infty}^2\kh{p+1} N\sum_{i=m+1}^{\infty}2^{i-1}2^{-2\kh{p+1}i}\\
        &= c_p^2\normm{f^{\kh{p}}}_{\infty}^2 \frac{N\kh{p+1}}{\kh{2^{2p+2}-2}2^m}2^{-2pm}.
    \end{align*}
    Then notice for matrix $A$ there holds $\|A\|_2\leq\|A\|_F$, we have the desired inequality. Worthy to notice that with large $m$, another inequality $\|A\|_2\leq\sqrt{\|A\|_1\|A\|_\infty}$ yields a better bound.
\end{proof}

Now we focus on constructing multiscale basis of $\bM_2^3$. Inspired by \cite{chen2002fast,chen1999construction}, here we define points $T_0\coloneqq\{t_i\DEF (i+1)/5:i=0,1,2,3\}$, which satisfies $T_0\subset\Phi(T_0)$. Then we construct the basis of $\bM_0^3$ and $\bW_1^3$ as follows. Define $W_0\DEF\{w_{0i}\}_{i=0}^3\subset\bM_0^3$ such that $w_{0i}(t_j)=\delta_{ij}$ for $i,j=0,1,2,3$, which results in
\begin{align*}
    w_{00}\kh{x}&=
        -\frac{125}{6}x^3+\frac{75}{2}x^2-\frac{65}{3}x+4,\\
    w_{01}\kh{x}&=
        \frac{125}{2}x^3-100x^2+\frac{95}{2}x-6,\\
    w_{02}\kh{x}&=
        -\frac{125}{2}x^3+\frac{175}{2}x^2-35x+4,\\
    w_{03}\kh{x}&=
        \frac{125}{6}x^3-25x^2+\frac{55}{6}x-1.
\end{align*}
Clear that $W_0$ forms a basis of $\bM_0^3$. For wavelet space $\bW_1^3$, we require the basis $W_1\DEF\{w_{1i}:i=0,1,2,3\}\subset\bM_1^3$ consisting of functions with vanishing moment $4$, that is, $(w_{1i},w_{0j})=0$ for $i,j=0,1,2,3$. One possible basis could be
\begin{align*}
    w_{10}\kh{x}&=\begin{cases}
        \frac{1}{48}\kh{-920x^3+1080x^2-320x+19},&0\leq x< \frac{1}{2},\\
        \frac{1}{48}\kh{7080x^3-15720x^2+11360x-2669},&\frac{1}{2}\leq x\leq 1,
    \end{cases}\\
    w_{11}\kh{x}&=\begin{cases}
        \frac{1}{48}\kh{-23480x^3+15720x^2-2700x+91},&0\leq x< \frac{1}{2},\\
        \frac{1}{48}\kh{520x^3-1080x^2+660x-101},&\frac{1}{2}\leq x\leq 1,
    \end{cases}\\
    w_{12}\kh{x}&=\begin{cases}
        \frac{1}{48}\kh{-520x^3+480x^2-60x-1},&0\leq x< \frac{1}{2},\\
        \frac{1}{48}\kh{23480x^3-54720x^2+41700x-10369},&\frac{1}{2}\leq x\leq 1,
    \end{cases}\\
    w_{13}\kh{x}&=\begin{cases}
        \frac{1}{48}\kh{-7080x^3+5520x^2-1160x+51},&0\leq x< \frac{1}{2},\\
        \frac{1}{48}\kh{920x^3-1680x^2+920x-141},&\frac{1}{2}\leq x\leq 1,
    \end{cases}
\end{align*}
therefore we have $W_1$ as a basis of $\bW_1^3$. Then basis $W_n$ for level $n\geq 2$ could be obtained via relation \cref{eq:multi_gen}.

 \bibliographystyle{elsarticle-num} 
 \bibliography{EMG_placement}





\end{document}